\documentclass{aic}
\aicAUTHORdetails{title = {Robust (Rainbow) Subdivisions and Simplicial Cycles}, 
  author = {Istv\'an Tomon},
  plaintextauthor = {Istvan Tomon},
  keywords = {Tur\'an problem, subdivision, rainbow, simplicial complex},
}   

\aicEDITORdetails{%
   year={2024},
   number={1},
   received={9 May 2022},   
   revised={24 April 2023},    
   published={8 January 2024},  
   doi={10.19086/aic.2024.1},      
}   

\usepackage{tikz}
\usetikzlibrary{positioning,arrows,shapes,decorations.markings,decorations.pathreplacing,matrix,patterns}
\tikzstyle{vertex}=[circle,draw=black,fill=black,inner sep=0,minimum size=5pt,text=white,font=\footnotesize]

\usepackage{graphicx}

\theoremstyle{plain}
\newtheorem{theorem}{Theorem}[section]

\newtheorem{claim}[theorem]{Claim}
\newtheorem{lemma}[theorem]{Lemma}

\theoremstyle{definition}
\newtheorem{definition}{Definition}

\newcommand{\wehp}{\textbf{w.e.h.p.}}

\begin{document}
	\begin{frontmatter}[classification=text]


\author[ist]{Istv\'an Tomon}

\begin{abstract}
		We present several results in extremal graph and hypergraph theory of topological nature. First, we show that if $\alpha>0$ and $\ell=\Omega(\frac{1}{\alpha}\log\frac{1}{\alpha})$ is an odd integer, then every graph $G$ with $n$ vertices and at least $n^{1+\alpha}$ edges contains an $\ell$-subdivision of the complete graph $K_t$, where $t=n^{\Theta(\alpha)}$. Also, this remains true if in addition the edges of $G$ are properly colored, and one wants to find a rainbow copy of such a subdivision. In the sparser regime, we show that properly edge colored graphs on $n$ vertices with average degree $(\log n)^{2+o(1)}$ contain rainbow cycles, while average degree $(\log n)^{6+o(1)}$ guarantees rainbow subdivisions of $K_t$ for any fixed $t$, thus improving recent results of Janzer and Jiang et al., respectively.  Furthermore, we consider certain topological notions of cycles in pure simplicial complexes (uniform hypergraphs). We show that if $G$ is a $2$-dimensional pure simplicial complex ($3$-graph) with $n$ $1$-dimensional and at least $n^{1+\alpha}$ 2-dimensional faces, then $G$ contains a triangulation of the cylinder and the M\"obius strip with $O(\frac{1}{\alpha}\log\frac{1}{\alpha})$ vertices. We present generalizations of this for higher dimensional pure simplicial complexes as well.
		
		In order to prove these results, we consider certain (properly edge colored) graphs and hypergraphs $G$ with strong expansion. We argue that if one randomly samples the vertices (and colors) of $G$ with not too small probability, then many pairs of vertices are connected by a short path whose vertices (and colors) are from the sampled set, with high probability.
\end{abstract}
\end{frontmatter}
	
	\section{Introduction}
	
	Extremal graph and hypergraph theory is one of the central areas of combinatorics. Given a family of $r$-uniform hypergraphs (or \emph{$r$-graphs}, for short) $\mathcal{H}$, one is interested in approximating the maximum number of edges in an $r$-graph $G$ on $n$ vertices which avoids a copy of every member of $\mathcal{H}$. This is called the \emph{extremal number} of $\mathcal{H}$. In this paper, we consider the extremal numbers of certain families that are topological in nature, and study so called rainbow variants as well.
	
	\subsection{Robust subdivisions}
	
	A fundamental result of Mader \cite{Mader67} from 1967 states that if $G$ is a graph on $n$ vertices, which contains no subdivision of the complete graph $K_t$, then $G$ has $O_{t}(n)$ edges. Bollob\'as and Thomason \cite{BoTho98}, and independently Koml\'os and Szemer\'edi \cite{KoSze96} further improved this by showing that such a graph has at most $O(t^2 n)$ edges. Since then, numerous variations of this problem have been considered. A particularly interesting direction is when one wants to control the size of the forbidden subdivision. Montgomery \cite{Mont15} showed that the same bounds hold if one forbids all subdivisions of $K_{t}$ of size at most $O_{t}(\log n)$. However, the situation changes if we forbid only constant sized subdivisions. Indeed, in a subdivision of $K_t$ (where $t\geq 3$) in which each edge is subdivided at most $\ell$ times, the girth is at most $3\ell+3$, and it is well known that there exist graphs with $n$ vertices, $n^{1+\Omega(1/\ell)}$ edges, and girth more than $3\ell+3$, see e.g. \cite{ErdSpenc74}. On the other hand, Kostochka and Pyber \cite{KostPyb88} proved that every graph with $n$ vertices and at least $4^{t^2}n^{1+\alpha}$ edges contains a subdivision of $K_t$ on at most $7t^2\log t/\alpha$ vertices, which answered a question of Erd\H{o}s \cite{Erdos71}. This was improved by Jiang \cite{Jiang11}, who showed that any graph with $n>n_0(t,\epsilon)$ vertices and $n^{1+\alpha}$ edges contains a subdivision of $K_t$, in which each edge is subdivided at most $O(1/\alpha)$ times. Further strengthenings were provided by Jiang and Seiver \cite{JiSe12}, and the current state of the art is due to O. Janzer \cite{Janzer21}: if $\ell$ is even, there exists $\epsilon=\epsilon(t,\ell)$ such that every graph with $n>n(\epsilon,t)$ vertices and $n^{1+1/\ell-\epsilon}$ edges contains a subdivision of $K_t$, in which each edge is subdivided exactly $\ell-1$ times. We will refer to this as an \emph{$(\ell-1)$-subdivision}. Simple probabilistic arguments show that this result is optimal up the value of $\epsilon$, which must tend to $0$ as $t$ tends to infinity.
	
	Most of the aforementioned results focused on the case where $t$ is a constant, and do not (seem to) apply if $t=t(n)$ is some rapidly growing function of $n$. In this direction, it follows from the celebrated dependent random choice method \cite{FoxSud11} that $n$-vertex dense graphs contain a $1$-subdivision of $K_t$ for some $t=\Omega(\sqrt{n})$. More precisely, Alon, Krivelevich and Sudakov \cite{AKS03} proved that any graph with $n$ vertices and at least $\delta n^{2}$ edges contains a 1-subdivision of $K_t$, where $t=\delta n^{1/2}/4$. See also \cite{FoxSud11} for a very short proof of a slightly weaker result. Note that this result is no longer meaningful if $G$ has $o(n^{3/2})$ edges, and simple probabilistic arguments show that for every $\epsilon>0$ there are graphs with $n^{3/2-\epsilon}$ edges containing no $1$-subdivision of $K_t$ for $t\geq t_0(\epsilon)$.
	
	In an attempt to unify many of the aforementioned results, we investigate the problem of finding $\ell$-subdivisions of $K_t$, where $\ell$ is constant and $t=t(n)$ might grow rapidly. But what kind of result can one hope for? Consider graphs $G$ with $n$ vertices and roughly $n^{1+\alpha}$ edges. If $G$ is the union of complete balanced bipartite graphs of size $n^{\alpha}$, then $G$ contains no subdivision of $K_t$ for $t>n^{\alpha/2}$. Also, there are graphs $G$ with $n$ vertices and $n^{1+\alpha}$ edges having girth $\Omega(1/\alpha)$. Therefore, a reasonable conjecture to make is that such graphs contain $\ell$-subdivisions of $K_t$ with $t=n^{\Theta(\alpha)}$ and $\ell=\Omega(1/\alpha)$. Our main theorem shows that this conjecture is almost true.
	
	\begin{theorem}\label{thm:robust}
		There exist constants $c_1,c_2>0$ such that the following holds. Let $\alpha\in (0,1/2)$, $\ell\geq \frac{c_1}{\alpha}\log\frac{1}{\alpha}$ odd, and $n>n_0(\alpha,\ell)$. Let $G$ be a graph on $n$ vertices with at least $n^{1+\alpha}$ edges. Then  $G$ contains the $\ell$-subdivision of $K_t$, where $t\geq n^{c_2\alpha}$.
	\end{theorem}
	
	Note that the condition of $\ell$ being odd is necessary, as the complete bipartite graph does not contain a subdivision of the triangle, where each edge is subdivided an even number of times. Also, in exchange for getting worse bounds on $\ell$, we can replace $t=n^{\Theta(\alpha)}$ with $t=n^{\alpha/2-\epsilon}$ for every $\epsilon>0$ (assuming $n$ is sufficiently large with respect to $\epsilon$ as well). See Theorem \ref{thm:robust_formal} for the formal statement.
	
	The proof of Theorem \ref{thm:robust} proceeds via studying a family of graphs, which we refer to as \emph{$\alpha$-maximal}. We say that a graph $G$ is $\alpha$-maximal if $G$ maximizes the quantity $e(H)/v(H)^{1+\alpha}$ among its subgraphs $H$. These graphs naturally appear in connection to degenerate Tur\'an problems: indeed, if one tries to prove that  every graph $G$ with $n$ vertices and $\Omega(n^{1+\alpha})$ edges contains some forbidden graph $F$, it is enough to prove this in case $G$ is $\alpha$-maximal, as we can always pass to the subgraph $H$ of $G$ maximizing $e(H)/v(H)^{1+\alpha}$. We remark that $\alpha$-maximality is a special case of $\psi$-maximality, which was introduced by Koml\'os and Szemer\'edi \cite{KoSze96}, however, we are unaware whether this instance was studied before.
	
	We show that if $G$ is  $\alpha$-maximal, then $G$ has extraordinary vertex- and edge-expansion properties. In particular, if one samples the vertices of $G$ with probability $d(G)^{-O(1)}$ (where $d(G)$ denotes the average degree), then with high probability, many pairs of vertices are connected by a short path, whose internal vertices are from the sample. A similar result remains true if the graph is properly edge colored, and we sample colors as well.
	
	\subsection{Rainbow cycles and subdivisions}
	
	A \emph{proper coloring} of the edges of a graph is a coloring, in which no two neighboring edges receive the same color. A \emph{rainbow copy} of a graph in an edge colored graph is a copy in which no two edges are colored with the same color. The study of rainbow Tur\'an problems was initiated by Keevash, Mubayi, Sudakov, and Verstra\"ete \cite{KMSV07}. The general question under consideration is that given a graph $H$, or more generally a family of graphs $\mathcal{H}$, at most how many edges can a properly edge colored graph on $n$ vertices have if it does not contain a rainbow copy of a member of $\mathcal{H}$. This is called the \emph{rainbow extremal number} of $\mathcal{H}$.
	
	Surprisingly, the rainbow extremal number of cycles are already not well understood. It is an elementary result in graph theory that every graph with $n$ vertices and $n$ edges contains a cycle. In contrast, there are properly edge colored graphs with $n$ vertices and average degree $\Omega(\log n)$ containing no rainbow cycles. Indeed, consider the graph of the hypercube $Q_m$, that is, $V(Q_m)=\{0,1\}^{m}$, and two vertices are joined by an edge if they differ in one coordinate. Coloring an edge with color $i$ if its end-vertices differ in coordinate $i$ gives a proper coloring with no rainbow cycles. On the other hand, Das, Lee, and Sudakov \cite{DLS13} proved that average degree $e^{(\log n)^{1/2+o(1)}}$ guarantees a rainbow cycle, which was further improved to $O((\log n)^{4})$ by O. Janzer \cite{Janzer20}. By studying properly edge colored $\alpha$-maximal graphs with $\alpha\approx 1/\log n$, we are able to get within a log factor of the lower bound.
	
	\begin{theorem}\label{thm:rainbow_cycle}
		Let $G$ be a properly edge colored graph with $n$ vertices containing no rainbow cycles. Then $G$ has at most  $n(\log n)^{2+o(1)}$ edges.
	\end{theorem}
	
	We remark that our method is very different from the approach of O. Janzer \cite{Janzer20}, it is closer in spirit to that of \cite{DLS13}.
	
	A rainbow variant of the forbidden subdivision problem was recently proposed by Jiang, Methuku, and Yepremyan \cite{JMY21}. They proved that if a properly edge colored graph $G$ on $n$ vertices contains no rainbow subdivision of $K_t$ for some fixed $t$, then $G$ has at most $ne^{O(\sqrt{\log n})}$ vertices. This upper bound was subsequently improved to $n(\log n)^{60}$ by Jiang, Letzter, Methuku, and Yepremyan \cite{JLMY21}. A simple consequence of our tools is the following further improvement.
	
	\begin{theorem}\label{thm:rainbow_subdivision}
		Let $t$ be a positive integer, and let $G$ be a properly edge colored graph with $n$ vertices containing no rainbow subdivision of $K_t$. Then $G$ has at most  $n(\log n)^{6+o(1)}$ edges.
	\end{theorem}

	Finally, we present a rainbow variant of Theorem \ref{thm:robust} as well.
	
	\begin{theorem}\label{thm:rainbow_robust}
		There exist constants $c_1,c_2>0$ such that the following holds. Let $\alpha\in (0,\frac{1}{2})$, $\ell\geq \frac{c_1}{\alpha}\log\frac{1}{\alpha}$ odd, and $n>n_0(\alpha,\ell)$. Let $G$ be a properly edge colored graph on $n$ vertices with at least $n^{1+\alpha}$ edges. Then  $G$ contains a rainbow $\ell$-subdivision of $K_t$, where $t\geq n^{c_2\alpha}$.
	\end{theorem}
	
	\subsection{Cycles in simplicial complexes}
	
	In the second part of our paper, we study extremal problems about cycles in hypergraphs. A celebrated result of Bondy and Simonovits \cite{BoSim74} states that if $G$ is a graph on $n$ vertices containing no cycle of length $2\ell$ for $\ell\geq 2$, then $G$ has $O(n^{1+1/\ell})$ edges. On the other hand, graphs avoiding cycles of odd length can have quadratically many edges. 
	
	Problems about extending these results to uniform hypergraphs are widely studied. It is already not clear how one defines cycles in $r$-uniform hypergraphs (or $r$-graphs, for short), and indeed, there are several different notions studied in the literature. A \emph{Berge cycle} of length $\ell$ is an $r$-graph consisting of $\ell$ edges $e_0,\dots,e_{\ell-1},e_{\ell}=e_{0}$ together with $\ell$ distinct vertices $v_1,\dots,v_{\ell}$  such that $v_{i}\in e_{i-1}\cap e_{i}$ for $i\in [\ell]$. The extremal numbers of Berge cycles are studied in \cite{BoGyo08,GyoLem12}, and are fairly well understood. A \emph{linear cycle} of length $\ell$ is an $r$-graph with  $\ell$ edges $e_0,\dots,e_{\ell-1},e_{\ell}=e_0$ such that in the circular order, any two consecutive edges intersect in exactly one vertex, and non-consecutive edges are disjoint. The extremal numbers of linear cycles are also well understood, see \cite{KMV15}. 
	
	Finally, a \emph{tight cycle} of length $\ell$ is a sequence of $\ell$ vertices $v_1,\dots,v_{\ell}$ together with the $\ell$ edges formed by $r$-tuples of consecutive vertices in the circular order. Perhaps, this definition of a cycle is the most mysterious. Indeed, until recently, very little was known about its extremal numbers. It was an old conjecture of S\'os, see also \cite{Verst16}, that any $r$-graph with $n$ vertices and at least $\binom{n-1}{r-1}$ contains a tight cycle (of some length) if $n$ is sufficiently large. This was disproved in a strong sense by B. Janzer \cite{BJanzer20}, who showed that an $r$-graph on $n$ vertices can have $\Omega(n^{r-1}\log n/\log\log n)$ edges without containing a tight cycle. On the other hand, Sudakov and Tomon \cite{SudTom21} proved the upper bound $n^{r-1+o(1)}$, which was further improved by Letzter \cite{Letzter21} to $O(n^{r-1}(\log n)^5)$. However, none of these proofs seem to extend to give any bounds on the extremal numbers of the tight cycle of fixed length $\ell$. Verstra\"ete \cite{Verst16} proposed the conjecture that if $r$ divides $\ell$, then any $r$-graph with $n$ vertices containing no tight cycle of length $\ell$ can have at most $O(n^{r-1+2(r-1)/\ell})$ edges. The complete $r$-partite $r$-graph shows that if $r\nmid \ell$, then the extremal number is $\Theta(n^r)$.
	
	In this paper, we are interested in cycles from a topological perspective. An $r$-uniform hypergraph corresponds to the $(r-1)$-dimensional pure simplicial complex equal to the downset generated by the set of edges. In the case of graphs, cycles are exactly the homeomorphic copies of $S^1$, while subdivisions of $K_t$ are the homeomorphic copies of $K_{t}$. Consider $3$-uniform tight cycles. Interestingly, tight cycles of even length are homeomorphic to the cylinder $S^1\times B^1$, while tight cycles of odd length are homeomorphic to the M\"obius strip, see Figure \ref{fig:tightcycle}. This motivates the question that at most how many edges can a 3-uniform hypergraph on $n$ vertices have without containing a homeomorphic copy of the cylinder or M\"obius strip on $\ell$ vertices. Before we embark on this problem, let us discuss further related results.
	
	Another natural way to define cycles in $r$-graphs from a topological perspective is to consider homeomorphic copies of the sphere $S^{r-1}$. This problem was already considered in 1973 by Brown, Erd\H{o}s and S\'os \cite{BES73}, who showed that a 3-graph with $n$-vertices containing no homeomorphic copy of $S^2$ has at most $O(n^{5/2})$ edges, and this bound is the best possible. Recently,  Kupavskii, Polyanskii, Tomon, and Zakharov \cite{KPTZ} showed that the same bounds hold if we forbid homeomorphic copies of any fixed orientable surface. This answered a question of Linial \cite{Linial08,Linial18}, who in general proposed many topological problems about $r$-graphs. Long, Narayanan, and Yap \cite{LNY20} proved that for every $r$ there exists $\lambda=\lambda(r)>0$ such that if $H$ is an $r$-graph, then any $r$-graph on $n$ vertices avoiding homeomorphic copies of $H$ has at most $O_{H}(n^{r-\lambda})$ edges. Determining the optimal value of $\lambda(r)$ for $r\geq 3$ remains an interesting open problem.
	
	Let us get back to the problem of forbidding triangulations of the cylinder and the M\"obius strip.  In \cite{KPTZ}, Kupavskii, Polyanskii, Tomon, and Zakharov studied so called \emph{topological cycles}, which are special triangulations of the cylinder and the M\"obius strip. In Theorem 3.9, they proved that every $3$-graph with $n$ vertices and at least $n^{2+\alpha}$ edges contains a topological cycle on at most $O(1/\alpha)$ vertices. However, upon closer inspection of their proof, the topological cycle they find is always a triangulation of the cylinder. Roughly, the reason for this is that they reduce the problem of finding topological cycles to a problem about finding rainbow cycles in certain properly edge colored graphs. However, in this setting, even cycles correspond to cylinders, and odd cycles correspond to the M\"obius strip. As the extremal number of odd cycles is $\Omega(n^2)$, this method completely breaks in case one wants to forbid the M\"obius strip. See a formal discussion of this in Section \ref{sect:simplicial}. With a different approach, we overcome this obstacle.

	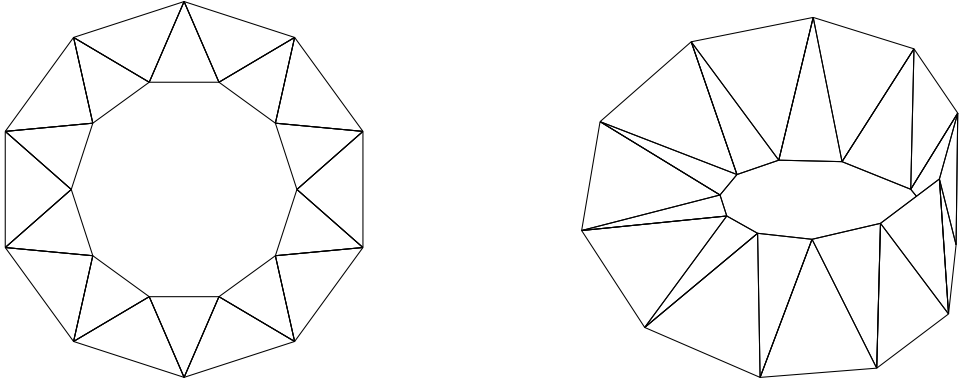
\begin{figure}
		\begin{minipage}[b]{0.45\textwidth}
			\begin{center}
				\begin{tikzpicture}
					\draw (4.87764,3.27254) -- (4,2.5) ;
\draw (4,2.5) -- (3.71353,3.38168) ;
\draw (4.87764,3.27254) -- (3.71353,3.38168) ;
\draw (3.71353,3.38168) -- (3.96946,4.52254) ;
\draw (4,2.5) -- (4.87764,3.27254) ;
\draw (4.87764,3.27254) -- (3.71353,3.38168) ;
\draw (3.71353,3.38168) -- (4.87764,3.27254) ;
\draw (4.87764,3.27254) -- (3.96946,4.52254) ;
\draw (3.96946,4.52254) -- (3.71353,3.38168) ;
\draw (3.71353,3.38168) -- (2.96353,3.92658) ;
\draw (3.96946,4.52254) -- (2.96353,3.92658) ;
\draw (2.96353,3.92658) -- (2.5,5) ;
\draw (3.71353,3.38168) -- (3.96946,4.52254) ;
\draw (3.96946,4.52254) -- (2.96353,3.92658) ;
\draw (2.96353,3.92658) -- (3.96946,4.52254) ;
\draw (3.96946,4.52254) -- (2.5,5) ;
\draw (2.5,5) -- (2.96353,3.92658) ;
\draw (2.96353,3.92658) -- (2.03648,3.92658) ;
\draw (2.5,5) -- (2.03648,3.92658) ;
\draw (2.03648,3.92658) -- (1.03054,4.52254) ;
\draw (2.96353,3.92658) -- (2.5,5) ;
\draw (2.5,5) -- (2.03648,3.92658) ;
\draw (2.03648,3.92658) -- (2.5,5) ;
\draw (2.5,5) -- (1.03054,4.52254) ;
\draw (1.03054,4.52254) -- (2.03648,3.92658) ;
\draw (2.03648,3.92658) -- (1.28647,3.38168) ;
\draw (1.03054,4.52254) -- (1.28647,3.38168) ;
\draw (1.28647,3.38168) -- (0.122359,3.27254) ;
\draw (2.03648,3.92658) -- (1.03054,4.52254) ;
\draw (1.03054,4.52254) -- (1.28647,3.38168) ;
\draw (1.28647,3.38168) -- (1.03054,4.52254) ;
\draw (1.03054,4.52254) -- (0.122359,3.27254) ;
\draw (0.122359,3.27254) -- (1.28647,3.38168) ;
\draw (1.28647,3.38168) -- (1,2.5) ;
\draw (0.122359,3.27254) -- (1,2.5) ;
\draw (1,2.5) -- (0.122358,1.72746) ;
\draw (1.28647,3.38168) -- (0.122359,3.27254) ;
\draw (0.122359,3.27254) -- (1,2.5) ;
\draw (1,2.5) -- (0.122359,3.27254) ;
\draw (0.122359,3.27254) -- (0.122358,1.72746) ;
\draw (0.122358,1.72746) -- (1,2.5) ;
\draw (1,2.5) -- (1.28647,1.61832) ;
\draw (0.122358,1.72746) -- (1.28647,1.61832) ;
\draw (1.28647,1.61832) -- (1.03054,0.477458) ;
\draw (1,2.5) -- (0.122358,1.72746) ;
\draw (0.122358,1.72746) -- (1.28647,1.61832) ;
\draw (1.28647,1.61832) -- (0.122358,1.72746) ;
\draw (0.122358,1.72746) -- (1.03054,0.477458) ;
\draw (1.03054,0.477458) -- (1.28647,1.61832) ;
\draw (1.28647,1.61832) -- (2.03647,1.07342) ;
\draw (1.03054,0.477458) -- (2.03647,1.07342) ;
\draw (2.03647,1.07342) -- (2.5,0) ;
\draw (1.28647,1.61832) -- (1.03054,0.477458) ;
\draw (1.03054,0.477458) -- (2.03647,1.07342) ;
\draw (2.03647,1.07342) -- (1.03054,0.477458) ;
\draw (1.03054,0.477458) -- (2.5,0) ;
\draw (2.5,0) -- (2.03647,1.07342) ;
\draw (2.03647,1.07342) -- (2.96352,1.07341) ;
\draw (2.5,0) -- (2.96352,1.07341) ;
\draw (2.96352,1.07341) -- (3.96946,0.477457) ;
\draw (2.03647,1.07342) -- (2.5,0) ;
\draw (2.5,0) -- (2.96352,1.07341) ;
\draw (2.96352,1.07341) -- (2.5,0) ;
\draw (2.5,0) -- (3.96946,0.477457) ;
\draw (3.96946,0.477457) -- (2.96352,1.07341) ;
\draw (2.96352,1.07341) -- (3.71352,1.61832) ;
\draw (3.96946,0.477457) -- (3.71352,1.61832) ;
\draw (3.71352,1.61832) -- (4.87764,1.72745) ;
\draw (2.96352,1.07341) -- (3.96946,0.477457) ;
\draw (3.96946,0.477457) -- (3.71352,1.61832) ;
\draw (3.71352,1.61832) -- (3.96946,0.477457) ;
\draw (3.96946,0.477457) -- (4.87764,1.72745) ;
\draw (4.87764,1.72745) -- (3.71352,1.61832) ;
\draw (3.71352,1.61832) -- (4,2.5) ;
\draw (4.87764,1.72745) -- (4,2.5) ;
\draw (4,2.5) -- (4.87764,3.27254) ;
\draw (3.71352,1.61832) -- (4.87764,1.72745) ;
\draw (4.87764,1.72745) -- (4,2.5) ;
\draw (4,2.5) -- (4.87764,1.72745) ;
\draw (4.87764,1.72745) -- (4.87764,3.27254) ;
				\end{tikzpicture}
			\end{center}
		\end{minipage}
		\begin{minipage}[b]{0.45\textwidth}
			\begin{center}
				\begin{tikzpicture}
					\draw (4.9779,1.86916) -- (4.87762,0.949325) ;
\draw (4.87762,0.949325) -- (4.75623,2.74503) ;
\draw (4.9779,1.86916) -- (4.75623,2.74503) ;
\draw (4.75623,2.74503) -- (5,3.62173) ;
\draw (4.87762,0.949325) -- (3.92287,0.228302) ;
\draw (3.92287,0.228302) -- (3.97303,2.15491) ;
\draw (4.87762,0.949325) -- (3.97303,2.15491) ;
\draw (3.97303,2.15491) -- (4.75623,2.74503) ;
\draw (3.92287,0.228302) -- (2.37245,0.105271) ;
\draw (2.37245,0.105271) -- (3.063,1.94275) ;
\draw (3.92287,0.228302) -- (3.063,1.94275) ;
\draw (3.063,1.94275) -- (3.97303,2.15491) ;
\draw (2.37245,0.105271) -- (0.838392,0.769736) ;
\draw (0.838392,0.769736) -- (2.3393,2.0225) ;
\draw (2.37245,0.105271) -- (2.3393,2.0225) ;
\draw (2.3393,2.0225) -- (3.063,1.94275) ;
\draw (0.838392,0.769736) -- (0,2.05924) ;
\draw (0,2.05924) -- (1.9294,2.25278) ;
\draw (0.838392,0.769736) -- (1.9294,2.25278) ;
\draw (1.9294,2.25278) -- (2.3393,2.0225) ;
\draw (0,2.05924) -- (0.245109,3.50784) ;
\draw (0.245109,3.50784) -- (1.83969,2.53233) ;
\draw (0,2.05924) -- (1.83969,2.53233) ;
\draw (1.83969,2.53233) -- (1.9294,2.25278) ;
\draw (0.245109,3.50784) -- (1.45662,4.57061) ;
\draw (1.45662,4.57061) -- (2.06515,2.80433) ;
\draw (0.245109,3.50784) -- (2.06515,2.80433) ;
\draw (2.06515,2.80433) -- (1.83969,2.53233) ;
\draw (1.45662,4.57061) -- (3.07835,4.89473) ;
\draw (3.07835,4.89473) -- (2.62084,2.99536) ;
\draw (1.45662,4.57061) -- (2.62084,2.99536) ;
\draw (2.62084,2.99536) -- (2.06515,2.80433) ;
\draw (3.07835,4.89473) -- (4.41822,4.47685) ;
\draw (4.41822,4.47685) -- (3.46163,2.97567) ;
\draw (3.07835,4.89473) -- (3.46163,2.97567) ;
\draw (3.46163,2.97567) -- (2.62084,2.99536) ;
\draw (4.41822,4.47685) -- (5,3.62173) ;
\draw (5,3.62173) -- (4.37505,2.60803) ;
\draw (4.41822,4.47685) -- (4.37505,2.60803) ;
\draw (4.37505,2.60803) -- (3.46163,2.97567) ;
\draw (5,3.62173) -- (4.9779,1.86916) ;
\draw (4.45095,2.515) -- (4.37505,2.60803) ;
\draw (4.9779,1.86916) -- (4.75623,2.74503) ;
\draw (4.75623,2.74503) -- (4.87762,0.949325) ;
\draw (4.9779,1.86916) -- (5,3.62173) ;
\draw (5,3.62173) -- (4.75623,2.74503) ;
\draw (4.87762,0.949325) -- (3.97303,2.15491) ;
\draw (3.97303,2.15491) -- (3.92287,0.228302) ;
\draw (4.87762,0.949325) -- (4.75623,2.74503) ;
\draw (4.75623,2.74503) -- (3.97303,2.15491) ;
\draw (3.92287,0.228302) -- (3.063,1.94275) ;
\draw (3.063,1.94275) -- (2.37245,0.105271) ;
\draw (3.92287,0.228302) -- (3.97303,2.15491) ;
\draw (3.97303,2.15491) -- (3.063,1.94275) ;
\draw (2.37245,0.105271) -- (2.3393,2.0225) ;
\draw (2.3393,2.0225) -- (0.838392,0.769736) ;
\draw (2.37245,0.105271) -- (3.063,1.94275) ;
\draw (3.063,1.94275) -- (2.3393,2.0225) ;
\draw (0.838392,0.769736) -- (1.9294,2.25278) ;
\draw (1.9294,2.25278) -- (0,2.05924) ;
\draw (0.838392,0.769736) -- (2.3393,2.0225) ;
\draw (2.3393,2.0225) -- (1.9294,2.25278) ;
\draw (0,2.05924) -- (1.83969,2.53233) ;
\draw (1.83969,2.53233) -- (0.245109,3.50784) ;
\draw (0,2.05924) -- (1.9294,2.25278) ;
\draw (1.9294,2.25278) -- (1.83969,2.53233) ;
\draw (0.245109,3.50784) -- (2.06515,2.80433) ;
\draw (2.06515,2.80433) -- (1.45662,4.57061) ;
\draw (0.245109,3.50784) -- (1.83969,2.53233) ;
\draw (1.83969,2.53233) -- (2.06515,2.80433) ;
\draw (1.45662,4.57061) -- (2.62084,2.99536) ;
\draw (2.62084,2.99536) -- (3.07835,4.89473) ;
\draw (1.45662,4.57061) -- (2.06515,2.80433) ;
\draw (2.06515,2.80433) -- (2.62084,2.99536) ;
\draw (3.07835,4.89473) -- (3.46163,2.97567) ;
\draw (3.46163,2.97567) -- (4.41822,4.47685) ;
\draw (3.07835,4.89473) -- (2.62084,2.99536) ;
\draw (2.62084,2.99536) -- (3.46163,2.97567) ;
\draw (4.41822,4.47685) -- (4.37505,2.60803) ;
\draw (4.37505,2.60803) -- (5,3.62173) ;
\draw (4.41822,4.47685) -- (3.46163,2.97567) ;
\draw (3.46163,2.97567) -- (4.37505,2.60803) ;
\draw (5,3.62173) -- (4.37505,2.60803) ;
\draw (4.37505,2.60803) -- (4.45095,2.515) ;
				\end{tikzpicture}
			\end{center}
		\end{minipage}
		\caption{A 3-uniform tight cycle of length 20 (left) and a tight cycle of length 21 (right).}
		\label{fig:tightcycle}
	\end{figure}

	\begin{theorem}\label{thm:maincycle0}(Informal.)
		Let $\alpha>0$, and let $G$ be a $3$-graph on $n$ vertices with at least $n^{2+\alpha}$ edges. If $n\geq n_0(\alpha)$, then $G$ contains a triangulation of the cylinder and the M\"obius strip on $O(\frac{1}{\alpha}\log\frac{1}{\alpha})$ vertices. 
	\end{theorem}
	
	We remark that a simple probabilistic argument shows that there are 3-graphs $G$ with $n$ vertices and more than $n^{2+\alpha}$ edges such that every triangulation of the cylinder and the  M\"obius strip in $G$ has at least $1/\alpha$ vertices. See Lemma \ref{lemma:construction} for a detailed argument.
	
	Instead of proving Theorem \ref{thm:maincycle0}, we prove a result which applies to much sparser 3-graphs as well. Given a 3-graph $G$, let $p(G)$ denote the number of $1$-dimensional faces of $G$, that is, the number of pairs of vertices which appear in an edge. As proved by Letzter \cite{Letzter21}, slightly improving the result of Sudakov and Tomon \cite{SudTom21}, if a 3-graph $G$ satisfies $e(G)=\Omega(p(G)(\log p(G))^{5})$, then $G$ contains a tight cycle. A similar strengthening of Theorem \ref{thm:maincycle0} also holds.
	
	\begin{theorem}\label{thm:maincycle1}(Informal.)
			Let $\alpha>0$, and let $G$ be a $3$-graph satisfying $e(G)>p(G)^{1+\alpha}$. If $p(G)\geq p_0(\alpha)$, then $G$ contains a triangulation of the cylinder and the M\"obius strip on $O(\frac{1}{\alpha}\log\frac{1}{\alpha})$ vertices. 
	\end{theorem}
	
	Clearly, this theorem implies Theorem \ref{thm:maincycle0} noting that $p(G)\leq \binom{n}{2}$.  For the formal version of Theorem \ref{thm:maincycle1}, as well as generalizations for $r$-graphs, see Section \ref{sect:simplicial}, Theorem \ref{thm:cycleformal} and Theorem \ref{thm:final}. In order to prove Theorem \ref{thm:maincycle1}, we study an analogue of $\alpha$-maximality for hypergraphs. Here, we say that $G$ is \emph{$\alpha$-maximal} if $G$ maximizes the quantity $e(H)/p(H)^{1+\alpha}$ among its subhypergraphs $H$. We show that $\alpha$-maximal hypergraphs have unusually good expansion properties. In the case of hypergraphs, it is already not clear how one defines expansion, and indeed, this has become a popular topic with several different definitions, see the survey of Lubotzky \cite{Lubotzky18}. However, most of these definitions are quite technical, and so far lacking combinatorial applications. Our notion of expansion is simple, we only require that subsets of $1$-dimensional faces expand via 2-dimensional faces.
	
	In the concluding remarks, we present some applications of our formal Theorem \ref{thm:cycleformal} about the extremal numbers of cycles in the hypercube.
	
	\bigskip

	\noindent
	\textbf{Organization of the paper.} We present our graph theoretic notation and some probabilistic tools in the next subsection, and introduce $\alpha$-maximal graphs in Section \ref{sect:alpha}. Then, we prepare and prove Theorems \ref{thm:rainbow_cycle}, \ref{thm:rainbow_subdivision}, and \ref{thm:rainbow_robust} in the rest of Section \ref{sect:rainbow}. In Section \ref{sect:introductiontosimplices}, we formally discuss simplicial complexes, hypergraphs, and introduce the notion of higher order walks, paths and cycles. We continue with the definition and properties of $\alpha$-maximal hypergraphs in Section \ref{sect:alphasimplicial}. We present the proof of Theorem \ref{thm:maincycle0} in Sections \ref{sect:simp_expansion} and \ref{sect:simpcycles}, and then we conclude our paper with some remarks.

	\section{Rainbow cycles and subdivisions}\label{sect:rainbow}
	
	\subsection{Preliminaries}
	Let us introduce our notation, which is mostly conventional. We omit floors and ceilings whenever they are not crucial. Given a graph $G$, $v(G)=|V(G)|,e(G)=|E(G)|$ denotes the number of vertices and edges of $G$, respectively, and $d(G)=2e(G)/v(G)$ is the average degree of $G$. Also, for $X\subset V(G)$, $N_{G}(X)=N(X)=\{y\in V(G)\setminus X:\exists x\in X, xy\in E(G)\}$ is the neighbourhood of $X$, and if $v\in V(G)$, then $N_{G}(v)=N(v):=N(\{v\})$ and $\deg_{G}(v)=\deg(v):=|N(v)|$. Furthermore, $G[X]$ denotes the subgraph of $G$ induced by $X$. If $Y\subset V(G)\setminus X$, $E_{G}(X,Y)=E(X,Y)$ is the set of edges between $X$ and $Y$, and $e_{G}(X,Y)=e(X,Y):=|E(X,Y)|$.
	
	Now let us collect a few technical results that we will be use repeatedly throughout the paper. The first is a variant of the Chernoff bound, see e.g. Theorem 1.1 in \cite{prob_book}.
	
	\begin{lemma}[Multiplicative Chernoff bound]
		Let $X_1,\dots,X_{n}$ be independent random variables taking values from $[0,1]$, let $X=\sum_{i=1}^{n}X_{i}$ and $\mu=\mathbb{E}(X)$. Then 
		$$\mathbb{P}\left(X\leq \frac{\mu}{2}\right)\leq e^{-\mu/8},$$
		and if $t>2\mu$, then 
		$$\mathbb{P}\left(X\geq t\right)\leq e^{-t/6},$$
	\end{lemma}

    The following estimates will also come in handy. Each can be verified using simple calculus.
    
    \begin{claim}\label{claim:technical}
    \begin{itemize}
        \item[(i)]   Let $a,b>0$ such that $ab<1$, then $(1-a)^{b}\leq e^{-ab}\leq 1-\frac{ab}{2}$. Furthermore, if $a\leq 1/2$ holds as well, then $1-2ab<(1-a)^b$
        \item[(ii)]  Let $a\in (0,1/2)$. Then $(1+a)^{1/(1+a)}>1+\frac{a}{2}$ and $(1+\frac{a}{2})^{1/(1+a)}\geq 1+\frac{a}{4}$.
    \end{itemize}
  
    \end{claim}

   We will make use of the following graph theoretic concentration inequality. Roughly, it says that given a bipartite graph, if we randomly sample the vertices in one of the parts, the size of the neighbourhood of the sample is very unlikely to be much smaller than its expected value.

    \begin{lemma}\label{lemma:main_technical}
        Let $p\in (0,1]$, and $\lambda>1$. Let $G$ be a bipartite graph with vertex classes $A$ and $B$. Let $U\subset A$ be a random sample, each vertex included independently with probability $p$.  Let $\mu:=\mathbb{E}(|N(U)|)$, and 
        suppose that every vertex in $A$ has degree at most $K$. If $K\leq \frac{\mu}{32\lambda \log_2(\lambda p^{-1})}$, then
        $$\mathbb{P}\left(|N(U)|\leq \frac{\mu}{64\lambda\log_2(\lambda p^{-1})}\right)\leq 2e^{-\lambda}.$$
    \end{lemma}

    \begin{proof}
         Let $\Delta:=\frac{\lambda}{p}$, $B_0=\{v\in B:\deg(v)\geq \Delta\}$, and $B_1=B\setminus B_0$. For $v\in B$, let $X_v$ be the indicator random variable of the event $v\in N(U)$. Then $\mu=\sum_{v\in B}\mathbb{E}(X_v)$. Consider two cases.
\medskip

        \noindent
         \textbf{Case 1.} $\sum_{v\in B_0}\mathbb{E}(X_v)\geq \frac{\mu}{2}$.

        Note that this implies $|B_0|\geq \frac{\mu}{2}$ as well. For $v\in B_0$, we have 
        $$\mathbb{E}(X_v)=1-(1-p)^{\deg(v)}\geq 1-e^{-p\cdot \deg(v)}\geq 1-e^{-\lambda}.$$ Hence, $\mathbb{E}(|B_0\setminus N(U)|)\leq |B_0|e^{-\lambda}$. Now apply Markov's inequality to conclude that 
        $$\mathbb{P}\left(|B_0\setminus N(U)|\geq \frac{|B_0|}{2}\right)\leq 2e^{-\lambda}.$$
        But then
        $$\mathbb{P}\left(|N(U)|\leq \frac{\mu}{4}\right)\leq \mathbb{P}\left(|B_0\setminus N(U)|\geq \frac{|B_0|}{2}\right)\leq  2e^{-\lambda}.$$
        This concludes Case 1.

        \bigskip
        
        \noindent
         \textbf{Case 2.} $\sum_{v\in B_1}\mathbb{E}(X_v)>\frac{\mu}{2}$.
         
         Using that $p\cdot \deg(v)\leq \lambda$, we can write
         $$p\cdot \deg(v)\geq \mathbb{E}(X_v)\geq 1-(1-p)^{\deg(v)/\lambda}\geq \frac{p\cdot \deg(v)}{2\lambda},$$
         where the last inequality holds by Claim \ref{claim:technical}. Observe that $\sum_{v\in B_1}\deg(v)=e_G(A,B_1)=:t$, so summing the previous inequality for every $v\in B_1$ implies that 
         $$p\cdot t\geq \frac{\mu}{2}\geq \frac{p\cdot t}{4\lambda}.$$     
         Here the last inequality holds by noting that $\sum_{v\in B_1}\mathbb{E}(X_v)\leq \mu$. For $i=0,\dots,\lfloor \log_2 \Delta\rfloor$, let 
         $$C_i=\{v\in B:2^{i}\leq \deg(v)<2^{i+1}\}.$$
         Then there exists $I\in\{0,\dots, \lfloor \log_2 \Delta\rfloor\}$ such that the bipartite graph $e_{G}(A,C_I)$ contains at least $t/2\log_2 \Delta$ edges. Set $k:=2^{I}$, then $|C_I|\geq \frac{t}{4k\log_2 \Delta}$. 
         
         Let $s:=\frac{2\Delta}{k}$. Next, we will construct disjoint sets $V_1,\dots,V_{\ell}\subset C_I$, each of size at most $s$ such that $V_1,\dots,V_{\ell}$ cover at least half of the elements of $C_I$, and $|N(V_i)|\geq \Delta$ for $i\in [\ell]$. Suppose that we already constructed $V_1,\dots,V_j$, and let $C=C_I\setminus (V_1\cup\dots\cup V_j)$. If $|C|\leq |C_I|/2$, we stop, otherwise let $V$ be a set of randomly chosen $s$ elements (with repetition) of $C$. For a vertex $w\in A$, let $d_w$ denote its degree in $G[A\cup C]$. Then we have $\mathbb{P}(w \in N(V))=1-\left(1-\frac{d_w}{|C|}\right)^{s}$. Recall that every vertex in $A$ has degree at most $K$, so
         \begin{equation}\label{equ:long}
            sd_w\leq s K \leq \frac{s\mu}{32\lambda \log_2(\Delta)}=\frac{\Delta\mu}{16k\lambda \log_2(\Delta)}=\frac{\mu}{16kp\log_2(\Delta)} \leq \frac{t}{8k\log_2(\Delta)}\leq \frac{|C_I|}{2}\leq |C|.
         \end{equation}
         Therefore, Claim \ref{claim:technical} implies $\mathbb{P}(w \in N(V))\geq \frac{sd_w}{2|C|}$. Hence, 
         $$\mathbb{E}(|N(V)|)\geq \sum_{w\in A}\frac{sd_w}{2|C|}=\frac{s\cdot e_G(A,C)}{2|C|}\geq \frac{sk}{2}=\Delta.$$ Here, the last inequality holds because  in $G[A\cup C]$, the degree of every vertex in $C$  is at least $k$. Thus, there exists a choice for $V$ such that $|N(V)|\geq \Delta$, set $V_{j+1}:=V$.

         For $i\in [\ell]$, let $Y_{i}$ be the indicator random variable of the event $N(U)\cap V_i\neq \emptyset$. Then, as in Case 1, we have $\mathbb{E}(Y_i)\geq 1-e^{-\lambda}$. Therefore, with probability at least $1-2e^{-\lambda}$, at least half of the indices $i\in [\ell]$ satisfy that $N(U)\cap V_i\neq \emptyset$, which implies that $|N(U)|\geq \frac{\ell}{2}$. This finishes Case 2. as 
         $$\frac{\ell}{2}\geq \frac{|C_I|}{4s}=\frac{|C_I|k}{8\Delta}\geq \frac{t}{32\Delta\log_2\Delta}\geq \frac{\mu}{64p\Delta\log_2\Delta}=\frac{\mu}{64\lambda\log_2\Delta}.$$ 
    \end{proof}

    Finally, we prove a variant of the previous lemma, in which one has a proper edge-coloring, and we sample colors as well. The proof is very similar to that of Lemma \ref{lemma:main_technical}, so we mostly highlight the differences. Given an edge coloring of a graph $G$, $Q$ is a subset of the colors, and $U\subset V(G)$, let $N_Q(U)$ denote the set of vertices $v\in V(G)\setminus U$ that are joined to some element of $U$ by an edge of color in $Q$.

    \begin{lemma}\label{lemma:main_technical_color}
        Let $p,p_c\in (0,1]$, and $\lambda>1$. Let $G$ be a bipartite graph with vertex classes $A$ and $B$, and let $f:E(G)\rightarrow R$ be a proper edge coloring. Let $U\subset A$ be a random sample of vertices, each vertex included independently with probability $p$, and let $Q\subset R$ be a random sample of colors, each included independently with probability $p_c$.  Let $\mu:=\mathbb{E}(|N_Q(U)|)$, and 
        suppose that every vertex in $A$ has degree at most $K$. If $K+|A|\leq \frac{\mu}{128\lambda \log_2(\lambda (p\cdot p_c)^{-1})}$, then
        $$\mathbb{P}\left(|N_Q(U)|\leq \frac{\mu}{64\lambda\log_2(\lambda (p\cdot p_c)^{-1})}\right)\leq 2e^{-\lambda}.$$
    \end{lemma}

    \begin{proof}
         Let $q=p\cdot p_c$, $\Delta=\frac{\lambda}{q}$,  $B_0=\{v\in B:\deg(v)\geq \Delta\}$, and $B_1=B\setminus B_0$. For $v\in B$, let $X_v$ be the indicator random variable of the event $v\in N_Q(U)$. Then $\mu=\sum_{v\in B}\mathbb{E}(X_v)$. Consider two cases.
\medskip

        \noindent
         \textbf{Case 1.} $\sum_{v\in B_0}\mathbb{E}(X_v)\geq \frac{\mu}{2}$.

         Identical to Case 1. in Lemma \ref{lemma:main_technical} with $p$ replaced by $q$.

          \bigskip
        
        \noindent
         \textbf{Case 2.} $\sum_{v\in B_1}\mathbb{E}(X_v)>\frac{\mu}{2}$.

         Repeating the same calculations as in Lemma \ref{lemma:main_technical}, we get 
         $$q\cdot t\geq \frac{\mu}{2}\geq \frac{q\cdot t}{4\lambda},$$
         where $t=e_G(A,B_1)$. Define $C_I$, $k$, $s=\frac{2\Delta}{k}$ as before, and recall that $|C_I|\geq \frac{t}{4k\log_2 \Delta}$. Given a set $V\subset C_I$, say that an edge in $G[A\cup V]$ is \emph{unique} (with respect to $V$) if no other edge in  $G[A\cup V]$ has the same color or same endpoint in $A$.
         
         We will construct disjoint sets $V_1,\dots,V_{\ell}\subset C_I$, each of size at most $s$ such that $V_1,\dots,V_{\ell}$ cover at least half of the elements of $C_I$, and $G[A\cup V_i]$ has at least $\Delta$ unique edges for $i\in [\ell]$.  Suppose that we already constructed $V_1,\dots,V_j$, and let $C=C_I\setminus (V_1\cup\dots\cup V_j)$. If $|C|\leq |C_I|/2$, we stop, otherwise let $V=\{v_1,\dots,v_s\}$ be a set of randomly chosen $s$ elements (with repetition) of $C$. Given an edge $\{u,v\}$ with $u\in A$, $v\in C$, the probability that $\{u,v\}$ is unique in $G[A\cup V]$ is at least 
         $$\frac{s}{|C|}\left(1-\frac{K+|A|}{|C|}\right)^{s-1}.$$
         Indeed, for every $i\in [s]$, we have $\mathbb{P}(v_i=v)=\frac{1}{|C|}$, and for $j\in [s]\setminus \{i\}$, the probability that $v_j$ is not a neighbour of $u$, or has an edge of color $f(\{u,v\})$, is at least $1-\frac{K+|A|}{|C|}$. Here, $\frac{(K+|A|)s}{|C|}\leq \frac{1}{4}$ by repeating  the same calculations as in (\ref{equ:long}). Hence, we can write 
         $$\frac{s}{|C|}\left(1-\frac{K+|A|}{|C|}\right)^{s-1}\geq \frac{s}{2|C|}.$$
         Therefore, the expected number of unique edges in $G[A\cup V]$ is at least $e(A,C)\cdot\frac{s}{2|C|}\geq \frac{sk}{2}\geq \Delta$. This implies that there is a choice for $V$ with at least $\Delta$ unique edges, set $V_{j+1}:=V$.

         For $i\in [\ell]$, let $Y_{i}$ be the indicator random variable that $N_Q(U)\cap V_i\neq \emptyset$. Then, as in Case 1, we have $\mathbb{E}(Y_i)\geq 1-e^{-\lambda}$. Therefore, with probability at least $1-2e^{-\lambda}$, at least half of the indices $i\in [\ell]$ satisfy that $N_Q(U)\cap V_i\neq \emptyset$, which implies that $|N_Q(U)|\geq \frac{\ell}{2}$. This finishes Case 2. by the same calculations as in Lemma \ref{lemma:main_technical}.
    \end{proof}

	\subsection{\texorpdfstring{$\alpha$}{a}-maximal graphs}\label{sect:alpha}
	In this section, we introduce \emph{$\alpha$-maximal} graphs and study their properties.
	
	\begin{definition}(\emph{$\alpha$-maximal graphs})
		Let $\alpha>0$. A graph $G$ is a \emph{$\alpha$-maximal} if for every subgraph $H$ of $G$, we have $$\frac{d(H)}{v(H)^{\alpha}}\leq \frac{d(G)}{v(G)^{\alpha}}.$$ Equivalently, if $e(G)=cv(G)^{1+\alpha}$, then $e(H)\leq cv(H)^{1+\alpha}$ for every subgraph $H$ of $G$.
	\end{definition}
	
	An obvious, but highly useful property of $\alpha$-maximal graphs is that every graph contains one. Indeed, the subgraph $H$ maximizing the quantity $d(H)/v(H)^{\alpha}$ is $\alpha$-maximal. Moreover, if a graph $G$ has $n$ vertices and average degree at least $cn^{\alpha}$, then it contains an $\alpha$-maximal subgraph $H$ of average degree at least $c$, and thus of size at least $c$. In the next lemma, we list a number of further properties of $\alpha$-maximal graphs. In particular, we show that they have large minimum degree, and that they are excellent vertex- and edge-expanders. 
	\begin{lemma} \label{lemma:maximal}\emph{(Properties of $\alpha$-maximal graphs)}
		Let $0<\alpha<\frac{1}{2}$, let $G$ be an $\alpha$-maximal graph on $n$ vertices, and let $d(G)=cn^{\alpha}$.
		\begin{itemize}
			\item[(i)] If $G$ is nonempty, then $c\geq 1/2$.
			\item[(ii)] The minimum degree of $G$ is at least $d(G)/2=cn^{\alpha}/2$.
		\end{itemize}
		Let $X\subset V(G)$ such that $|X|\leq n/2$, and let $Y=V(G)\setminus X$.
		\begin{itemize}
			\item[(iii)] $e(X,N(X))\geq \frac{c}{4}n^{\alpha}|X|\left(1+\alpha-\left(\frac{|X|}{|Y|}\right)^{\alpha}\right)$.
			\item[(iv)] $|N(X)|>|X|\left(\left(1+\frac{\alpha}{2}\right)\left(\frac{|Y|}{|X|}\right)^{\alpha/(1+\alpha)}-1\right).$
		\end{itemize}
	\end{lemma}
	
	\begin{proof}
		\begin{itemize}
			\item[(i)]If $H$ is the subgraph formed by a single edge, then $d(H)/v(H)^{\alpha}=1/2^{\alpha}>1/2$.
			\item[(ii)]Let $\delta$ be the minimum degree of $G$, and let $v$ be a vertex of degree $\delta$. Let $H=G$ be the subgraph of $G$ we get after removing $v$, then by the definition of $\alpha$-maximal, we have $d(H)/v(H)^{\alpha}\leq d(G)/v(G)^{\alpha}$. This can be rewritten as $$\frac{d(G)n-2\delta}{(n-1)^{1+\alpha}}\leq \frac{d(G)}{n^{\alpha}},$$
			which gives
			$$\frac{d(G)}{2}\left(n-\frac{(n-1)^{1+\alpha}}{n^{\alpha}}\right)\leq \delta,$$
			implying the desired inequality $\delta\geq d(G)/2$.
			\item[(iii)] We have
			$$e(X,N(X))=e(G[X\cup Y])-e(G[X])-e(G[Y])\geq \frac{c}{2}((|X|+|Y|)^{1+\alpha}-|X|^{1+\alpha}-|Y|^{1+\alpha}) .$$ 
			Using the inequalities  $(1+|X|/|Y|)^{1+\alpha}\geq 1+(1+\alpha)|X|/|Y|$ and $|Y|^{\alpha}\geq n^{\alpha}/2$, the right hand side is at least
			$$\frac{c}{2}((1+\alpha)|X||Y|^{\alpha}-|X|^{1+\alpha})\geq \frac{c\alpha}{4}|X|n^{\alpha}+\frac{c}{4}|X|n^{\alpha}\left(1-\left(\frac{|X|}{|Y|}\right)^{\alpha}\right).$$
			\item[(iv)]  We can write
			$$e(G[X\cup N(X)])\geq e(G[X\cup Y])-e(G[Y])\geq \frac{c}{2}((|X|+|Y|)^{1+\alpha}-|Y|^{1+\alpha}).$$
			As in (iii), we can further bound the right hand side from below by 
			$\frac{c(1+\alpha)}{2}|X||Y|^{\alpha}.$
			On the other hand, $e(G[X\cup N(X)])\leq \frac{c}{2}(|X|+|N(X)|)^{1+\alpha}$. Comparing the lower and upper bound on $e(G[X\cup N(X)])$, we get
			$$|N(X)|\geq \left((1+\alpha)|X||Y|^{\alpha}\right)^{1/(1+\alpha)}-|X|\geq |X|\left(\left(1+\frac{\alpha}{2}\right)\left(\frac{|Y|}{|X|}\right)^{\alpha/(1+\alpha)}-1\right).$$
            The last inequality follows by Claim \ref{claim:technical}.
		\end{itemize}
	\end{proof}
	
	Note that by taking $C$ to be a small constant, every graph $G$ on $n$ vertices of average degree $d$ contains a $C/(\log n)$-maximal subgraph $H$ of average degree at least $d(1-\epsilon)$. The graph $H$ satisfies that if $X\subset V(H)$ and $|X|\leq v(H)/2$, then $|N(X)|=\Omega(|X|/\log n)$. Graphs with similar expansion properties commonly appear in the study of sparse extremal problems, see e.g. \cite{JMY21,Mont15,ShaSu15,SudTom21}. Our definition immediately implies their existence in every graph with positive constant average degree, unlike earlier arguments.
	
	\subsection{Colorful expansion of random samples}
	
	In this section, we prepare the main technical tools needed to prove our rainbow results. To give some motivation, let us briefly outline the proof of Theorem \ref{thm:rainbow_subdivision}, as the proof of our other results follow on a very similar line.
	
	\bigskip
	
	\noindent
	\textbf{Outline of the proof of Theorem \ref{thm:rainbow_subdivision}.} Let $G$ be a graph on $n$ vertices of average degree at least $d=(\log n)^{6+\epsilon}$, and let $f:E(G)\rightarrow R$ be a proper edge coloring. We want to prove that $G$ contains rainbow subdivision of $K_{t}$. Let $\alpha=1/\log n$, then $G$ contains an $\alpha$-maximal subgraph $H$ with average degree $\Omega(d)$. Clearly, it is enough to show that $H$ contains a rainbow subdivision of $K_t$. One of our main technical results, Lemma  \ref{lemma:randompath}, tells us that if one samples the vertices and colors of $H$ with probability $p=(\log n)^{-1-\epsilon/10}$, and $v\in V(H)$, then with high probability at least $1/3$ proportion of the vertices of $H$ can be reached from $v$ by a rainbow path of length at most $\ell=(\log n)^{1+o(1)}$, whose internal vertices are from the vertex sample, and whose edges are colored with the sampled colors. By considering a random partition of the vertices and colors into $s=1/p$ parts, we argue that at least $\Omega(v(H)^{2})$ pairs of vertices $(x,y)$ have the property that there are at least $\Omega(s)$ internally vertex disjoint paths of length at most $\ell$ between $x$ and $y$, such that no color appears twice in the union of these paths. Let $L$ be the graph whose edges are such pairs of vertices. We are almost done. We find a 1-subdivision of $K_{t}$ in $L$, and greedily replace each edge $xy$ of this subdivision by a path of $H$ between $x$ and $y$ making sure that we avoid repeating vertices and colors. This is possible as $\ell t^{2}=o(s)$.
	
	\bigskip
	
	Most of the work needed to prove our theorems is concentrated in the following technical lemma. It roughly says that if $G$ is $\alpha$-maximal, $B\subset V(G)$, and $U$ is a random sample of $B$, where each vertex is sampled with some not too small probability, then $U$ expands as well as one would expect any $|B|$ element subset of $V(G)$ to expand.
	
	Let us introduce some notation. Let $G$ be a graph and $f:E(G)\rightarrow R$ be a coloring of the edges, where $R$ is some finite set. Given $U\subset V(G)$ and $Q\subset R$, a $(U,Q)$-path in $G$ is a rainbow path in which every internal vertex is contained in $U$, and every edge is colored with a color from $Q$.  Furthermore, let $\phi:V(G)\rightarrow 2^{V(G)\cup R}$ be a function that assigns every vertex a set of forbidden vertices and colors. For $X\subset V(G)$ and $Q\subset R$, define the \emph{restricted neighborhood with respect to the colors in $Q$} as $$N_{Q,\phi}(X)=\{y\in V(G)\setminus X: \exists x\in X, xy\in E(G), y\not\in \phi(x), f(xy)\in Q\setminus\phi(x)\}.$$

	\begin{lemma}\label{lemma:master}
		Let $p,p_c\in (0,1]$, $\alpha\in (0,\frac{1}{2})$, let $n$ be a positive integer and $\lambda>10^{10}\log \frac{2}{p\cdot p_c}$. Let $G$ be a graph on $n$ vertices with proper edge coloring $f:E(G)\rightarrow R$, and $B\subset V(G)$ satisfying the following properties.
		\begin{itemize}
			\item $G$ is $\alpha$-maximal,
			\item  $d:=d(G)\geq \lambda^3(\alpha\cdot p\cdot p_{c})^{-1}$,
			\item  $\phi:V(G)\rightarrow 2^{V(G)\cup R}$ such that $|\phi(v)|\leq d\alpha/32$ for every $v\in V(G)$,
			\item  $2\lambda^6 (p\cdot p_c)^{-1}<|B|<n/2$.
		\end{itemize}
		Let $U\subset B$ be a random sample of the vertices, each element chosen with probability $p$, and let $Q\subset R$ be a random sample of colors, each element chosen with probability $p_c$. Then with probability at least $1-2e^{-\lambda}$, $$|N_{Q,\phi}(U)\setminus B|\geq \frac{|B|}{4}\min\left\{\frac{d\cdot p\cdot p_c\cdot\alpha}{64\lambda^{5}},\left(\frac{n}{2|B|}\right)^{\alpha/(1+\alpha)}-1\right\}.$$
	\end{lemma}
	
	\begin{proof}
		Write $d=cn^{\alpha}$. Let $H$ be the bipartite graph with vertex classes $B$ and $N(B)$, in which $x\in B$ and $y\in N(B)$ are joined by an edge if $xy\in E(G)$, $y\not\in \phi(x)$ and $f(xy)\not\in \phi(x)$. Also, let $H_{Q}$ be the subgraph of $H$ in which we keep only the edges, whose color is in $Q$. Our goal is to show that $|N_{H_Q}(U)|$ is large with high probability.
		
		We are going to consider three cases, depending on the degree distribution of $H$. Let $\Lambda:=\frac{d\cdot p\cdot p_c\cdot \alpha|B|}{\lambda^3}$, and let $Y\subset B$ be the set of vertices $x$ such that $\deg_{H}(x)\geq \Lambda$. 
		
		\bigskip
		
		\noindent 
		\textbf{Case 1.} $|Y|\geq 2\lambda( p\cdot p_c)^{-1}$.
		
		Let $Z:=Y\cap U$, then $\mathbb{E}(|Z|)\geq 2\lambda p_{c}^{-1}$. Therefore, $$\mathbb{P}\left(|Z|\geq  \lambda p_c^{-1}\right)\geq 1-e^{-\Omega(\lambda)}$$ by the multiplicative Chernoff bound. 
		Assume that $|Z|\geq \lambda p_c^{-1}$, and let $Z_0$ be any $\lambda p_c^{-1}$ element subset of $Z$. Now let us sample the colors. Consider the bipartite graph $J$ with vertex classes $R$ and $N_H(Z_0)$, in which $r\in R$ is joined to $y\in N_H(Z_0)$ if there is an edge of color $r$ between $y$ and $Z_0$. Then every vertex in $J$ (in both parts) has degree at most $|Z_0|$, and $J$ has at least $\Lambda |Z_0|$ edges. Let us apply Lemma \ref{lemma:main_technical} to the bipartite graph $J$, with the following substitutions: $p_c\rightarrow p$, $\lambda\rightarrow\lambda$, $R\rightarrow A$, $N_H(Z_0)\rightarrow B$,  $Q\rightarrow U$, $|Z_0|\rightarrow K$. Given $y\in N_{H}(Z_0)$, 
        $$\mathbb{P}(y\in N_J(Q))=1-(1-p_c)^{|N_{J}(y)|}\geq 1- (1-p_c)^{|N_{J}(y)|/2\lambda}\geq \frac{p_c}{4\lambda}|N_{J}(y)|,$$
        where the last inequality holds by Claim \ref{claim:technical} and noting that $|N_J(Y)|\leq |Z_0|=\lambda p_c^{-1}$. Hence,  $\mathbb{P}(y\in N_J(Q))\geq \frac{p_c}{4\lambda}|N_{J}(y)|.$ But then
		$$\mathbb{E}(|N_J(Q)|)\geq \sum_{y\in N_{H}(Z_0)}\frac{p_c}{4\lambda}|N_{J}(y)|=\frac{p_c e(J)}{4\lambda}\geq \frac{p_c |Z_0|\Lambda}{4\lambda}=\frac{\Lambda}{4}.$$
        Therefore, we have $\mu:=\mathbb{E}(|N_J(Q)|)\geq\frac{\Lambda}{4}$. Furthermore, observe that 
        $$\frac{\Lambda}{|Z_0|}=\frac{d\cdot p\cdot p_c^2\cdot \alpha\cdot|B|}{\lambda^4}\geq \frac{1}{\lambda^4} p_c|B|\geq 8\lambda^2.$$
		That is, $|Z_0|\leq \frac{\mu}{32\lambda \log_2(\lambda p_c^{-1})}$ is satisfied (using our lower bound on $\lambda$), so we can apply Lemma \ref{lemma:main_technical} to conclude that
        $$\mathbb{P}\left(|N_J(Q)|\leq \frac{\mu}{64\lambda \log_2(\lambda p_c^{-1})}\right)\leq 2e^{-\lambda}.$$
        Noting that $\frac{\mu}{32\lambda\log_2(\lambda p_c^{-1})}\geq  \frac{|B|}{4}\cdot\frac{d\cdot p\cdot p_c\cdot\alpha}{64\lambda^{5}}$, this finishes Case 1.

		
		\bigskip
		In what follows, suppose that $|Y|\leq 2\lambda(p\cdot p_c)^{-1}$, and let $B':=B\setminus Y$. Then $|B'|\geq |B|/2$. Let $S\subset N_{H}(B')$ be the set of vertices $v$ such that $|N_{H}(v)\cap B'|\geq (p\cdot p_{c})^{-1}\lambda=:\Delta$, and let $T:=N_{G}(B)\setminus S$. 
		
		\medskip
		
		\noindent
		\textbf{Case 2.} $e_{G}(B',T)\leq d\alpha|B'|/16$  (note that we consider the number of edges in $G$ instead of $H$).
		
		Let $C=V(G)\setminus B'$. Noting that $|C|^{\alpha}\geq n^{\alpha}/2$, we also have $e_{G}(B',T)\leq c\alpha |B||C|^{\alpha}/8$. Crucially, we have $$E(G)=E(G[B'\cup S])\cup E(G[C])\cup E(G[B',T])\cup E(G[B',Y]).$$
		Therefore, we can write
		\begin{align*}
			e_{G}(B'\cup S)&\geq e(G)-e(G[C])-e_{G}(B',T)-e_{G}(B',Y)\\
			&\geq \frac{c}{2}((|B'|+|C|)^{1+\alpha}-|C|^{1+\alpha})-e_{G}(B',T)-e_G(B',Y)\\
			&\geq \frac{c}{2}(1+\alpha)|B'||C|^{\alpha}-\frac{c\alpha}{8}|B'||C|^{\alpha}-|Y||B'|\\
			&\geq \frac{c}{2}\left(1+\frac{\alpha}{2}\right)|B'||C|^{\alpha}.
		\end{align*}
		Here, the second inequality holds by the $\alpha$-maximality of $G$, while the third inequality follows from
		$$(|B'|+|C|)^{1+\alpha}=|C|^{1+\alpha}\left(1+\frac{|B'|}{|C|}\right)^{1+\alpha}\geq |C|^{1+\alpha}+(1+\alpha)|B'||C|^{\alpha}.$$
		The fourth inequality is true as $|Y|\leq 2\lambda(p\cdot p_{c})^{-1}<d\alpha/16<c\alpha|C|^{\alpha}/8$.
		
		On the other hand, $e_{G}(B'\cup S)\leq \frac{c}{2}(|B'|+|S|)^{1+\alpha}$. Comparing the lower and upper bound, we get 
		$$|S|\geq \left(\left(1+\frac{\alpha}{2}\right)|B'||C|^{\alpha}\right)^{1/(1+\alpha)}-|B'|\geq \left(1+\frac{\alpha}{4}\right)|B'|\left(\frac{|C|}{|B'|}\right)^{\alpha/(1+\alpha)}-|B'|.$$
		Here, the second inequality follows from Claim \ref{claim:technical}. Therefore, it suffices to show that  $|N_{H_Q}(U\cap B')\cap S|\geq |S|/2$ with high probability. For $v\in S$ let $X_{v}$ be the indicator random variable of the event $N_{H_Q}(v)\cap U\cap B'=\emptyset$, and let $X=\sum_{v\in S}X_{v}$. Then $|N_{H_Q}(U\cap B')\cap S|=|S|-X$. We have 
		$$\mathbb{E}(X_{v})=(1-p\cdot p_{c})^{|N_{H}(v)\cap B'|}\leq (1-p\cdot p_{c})^{\Delta}\leq e^{-\lambda}.$$
		Therefore, $\mathbb{E}(X)\leq |S|e^{-\lambda}$, and by Markov's inequality, $\mathbb{P}(X\geq |S|/2)\leq 2e^{-\lambda}$. This finishes the second case.
		
		\bigskip
		
		\noindent
		\textbf{Case 3.} $e_{G}(B',T)> d\alpha|B'|/16$.
		
		Then $$e_{H}(B',T)\geq e_{G}(B',T)-\sum_{v\in B'}|\phi(v)|\geq \frac{d\alpha |B'|}{32}.$$

        
		For each $v\in T$, let $X_{v}$ be the indicator random variable of the event $v\in N_{H_Q}(U\cap B')$. Let $\deg(v)$ denote the degree of $v$ in the bipartite graph $H[B'\cup T]$, then using that $\deg(v)\leq \lambda (p\cdot p_c)^{-1}$, we have
		$$\mathbb{E}(X_{v})=1-(1-p\cdot p_c)^{\deg(v)}\geq 1-(1-p\cdot p_c)^{\deg(v)/\lambda}\geq \frac{1}{2\lambda}\cdot p\cdot p_c\cdot \deg(v),$$
        where the last inequality holds by Claim \ref{claim:technical}.
		Let $X:=\sum_{v\in T} X_{v}$. Then 
		$$\mu:=\mathbb{E}(X)\geq \sum_{v\in T}\frac{1}{2\lambda}\cdot p \cdot p_{c}\cdot\deg(v)=\frac{p\cdot p_c\cdot e_{H}(B',T)}{2\lambda}\geq \frac{p\cdot p_{c}\cdot d\cdot\alpha |B'|}{64\lambda}.$$
        Let us apply Lemma \ref{lemma:main_technical_color} with the following assignments: $H[B'\cup T]\rightarrow G$, $p\rightarrow p$, $p_c\rightarrow p_c$, $\lambda\rightarrow\lambda$, $B'\rightarrow A$, $T\rightarrow B$,  $U\rightarrow U$, $Q\rightarrow Q$. The degree of every vertex in $B'$ is at most $K:=\Lambda$, so we have to verify that $\Lambda+|B'|\leq \frac{\mu}{64\lambda\log_2(\lambda (p\cdot p_c)^{-1})}$. But we have
        $$\Lambda=\frac{d\cdot p\cdot p_c\cdot \alpha|B|}{\lambda^3}\leq \frac{128\mu}{\lambda^2}\leq \frac{\mu}{128\lambda\log_2(\lambda (p\cdot p_c)^{-1})},$$
        where the last inequality holds by our lower bound on $\lambda$. Also,
        $$|B'|\leq \frac{64\lambda\mu}{\alpha\cdot p\cdot p_c\cdot d}\leq \frac{64\mu}{\lambda^2}\leq \frac{\mu}{128\lambda\log_2(\lambda (p\cdot p_c)^{-1})},$$
        where the second inequality holds by our lower bound on $d$, and the last inequality holds by our lower bound on $\lambda$. Therefore, the desired condition of Lemma \ref{lemma:main_technical_color} is satisfied, so we have
        $$\mathbb{P}\left(|N_{H_Q}(U)|\leq \frac{\mu}{64\lambda\log_2(\lambda (p\cdot p_c)^{-1})}\right)\leq 2e^{-\lambda}.$$
		Noting that $\frac{\mu}{64\lambda\log_2(\lambda (p\cdot p_c)^{-1})}\geq  \frac{|B|}{4}\cdot\frac{d\cdot p\cdot p_c\cdot\alpha}{64\lambda^{5}}$, this finishes Case 3. and the proof of the lemma.
	\end{proof}
	
	From this, we deduce that if one samples the vertices and colors of an $\alpha$-maximal graph with appropriate probabilities, then many vertices can be reached by a short rainbow path from a given vertex $v$, whose internal vertices and colors are from the sample. In the next lemma, $C$ will denote some unspecified, large absolute constant. More precisely, we prove that there exists some constant $C$ such that the following lemma is true.
	
	\begin{lemma}\label{lemma:randompath}
		Let $p,p_c\in (0,1]$, $\alpha\in (0,\frac{1}{2})$, $n$ be a positive integer, $\tau\in [1/\log_3 n,1/2)$ and $$\lambda>C(\log\log n)^{100}\log\frac{2}{p\cdot p_c}.$$ Let $G$ be an $\alpha$-maximal graph on $n$ vertices with proper edge coloring $f:E(G)\rightarrow R$, $d:=d(G)$, satisfying the following conditions. If $p=1$, then $d>C\lambda^{7}(\alpha^{2}\cdot p_c^{2})^{-1}n^{\alpha}$, otherwise $d>C\lambda^7(\alpha^{3}\cdot p\cdot p_c^{2})^{-1}n^{\alpha}$.
		
		Let $U$ be a random sample of vertices, each chosen with probability $p$, and let $Q\subset R$ be a random sample of the colors, each chosen with probability $p_c$. Then for every vertex $v\in V(G)$, with probability at least $1-O(\alpha^{-1}\log(1/\tau)e^{-\lambda})$, at least $n^{1-\tau}$ vertices of $G$ can be reached from $v$ by a $(U,Q)$-path of length at most $O(\alpha^{-1}\log (1/\tau))$. 
	\end{lemma}
	
	\begin{proof}
		Let $\ell=100\alpha^{-1}\log (1/\tau)$ (then $\ell\leq d\alpha/64$ is satisfied), and let $q,q_c\in (0,1]$ be the unique solutions of the equations $p=1-(1-q)^{\ell-1}$ and $p_c=1-(1-q_c)^{\ell}$, respectively. If $p=1$, then $q=1$. In general,  $q=\Omega(p/\ell)$ and  $q_c=\Omega(p_c/\ell)$. Instead of sampling the vertices with probability $p$, we sample the vertices in $\ell-1$ rounds, in each round with probability $q$, independently from the other rounds, and similarly with the colors (this method is sometimes referred to as ''sprinkling''). For $i=1,\dots,\ell-1$, let $U_{i}$ be a random sample of the vertices, each vertex chosen with probability $q$, and for $i=1,\dots,\ell$, let $Q_{i}$ be a random sample of $R$, each color chosen independently with probability $q_c$. Then $\bigcup_{i=1}^{\ell-1}U_i$ has the same distribution as $U$, and $\bigcup_{i=1}^{\ell}Q_i$ has the same distribution as $Q$.
		
		For $i=1,\dots,\ell-1$, let $B_i$ be the set of vertices that can be reached from $v$ by a $(U_1\cup\dots\cup U_{i-1},Q_{1}\cup\dots\cup Q_{i})$-path of length at most $i$. Also, let $\phi_{i}:V(G)\rightarrow 2^{V(G)\cup R}$ be any function that assigns every $x\in B_i$ the set of vertices and colors that appear on such a path, and let $\phi_{i}(x)=\emptyset$ if $x\not\in B_{i}$. Then $|\phi_{i}(x)|\leq 2i\leq 2\ell$ for every $x\in V(G)$. Our goal is to show that the sequence $|B_1|,\dots,|B_{\ell-1}|$ is rapidly increasing with high probability, thus $|B_{\ell-1}|\geq n^{1-\tau}$. 
		
		First, note that $B_1$ is the set of neighbours $x$ of $v$ such that $f(vx)\in Q$. Noting that $\deg(v)\geq d/2$ by Lemma \ref{lemma:maximal}, we have $\mathbb{E}(|B_1|)= \deg(v)q_{c}\geq dq_{c}/2$. Therefore, by the multiplicative Chernoff bound, we have $\mathbb{P}(|B_1|\geq dq_{c}/4)\geq 1-e^{-\Omega(dq_{c})}\geq 1-e^{-\lambda}$. In what comes, we assume that $|B_1|\geq dq_c/4>2\lambda^{6}(q\cdot q_c)^{-1}$.
		
		Note that $(N_{Q_{i},\phi_{i}}(B_{i}\cap U_{i})\cup B_{i})\subset B_{i+1}$, and
        $$d\geq \left\{\begin{array}{lr}
        C\lambda^7(\alpha^2\cdot p_c^2)^{-1}n^{\alpha}, & \text{if } p=1\\
        C\lambda^7(\alpha^3\cdot p\cdot p_c^2)^{-1}n^{\alpha}, & \text{if } p<1
        \end{array}\right\}\geq C'\lambda^7 (\alpha\cdot q\cdot q_c)^{-1}\left(\log\frac{1}{\tau}\right)^{-3}\geq \lambda^3(\alpha\cdot q\cdot q_c)^{-1}.$$ 
         Here, $C'$ denotes a large unspecified constant depending on $C$. Also, the second inequality holds by omitting a factor of $p_c^{-1}\geq 1$ and $n^{\alpha}\geq 1$, and recalling the bounds on $q$ and $q_c$. The inequality $|B_{i}|\geq |B_{1}|\geq 2\lambda^{6}(q\cdot q_c)^{-1}$ holds as well, so as long as $|B_{i}|\leq n/3$, the desired conditions of Lemma \ref{lemma:master} are satisfied (with $q$ instead of $p$, and $q_c$ instead of $p_c$). Therefore, we can conclude that
		$$|N_{Q_{i},\phi_{i}}(U_{i}\cap B_i)|\geq \frac{|B_i|}{4}\min\left\{\frac{d\cdot q\cdot q_c\cdot\alpha}{64\lambda^{5}},\left(\frac{n}{2|B_{i}|}\right)^{\alpha/(1+\alpha)}-1\right\}$$
		with probability at least $1-2e^{-\lambda}$. Note that the left side of the minimum is at least $n^{\alpha}$, so the minimum is always attained by the right side. Therefore, 
		$$|N_{Q_{i},\phi_{i}}(U_{i}\cap B_i)|\geq \frac{|B_i|}{4}\left(\left(\frac{n}{2|B_{i}|}\right)^{\alpha/(1+\alpha)}-1\right)\geq \frac{|B_{i}|}{4}\left(\left(\frac{n}{2|B_{i}|}\right)^{\alpha/2}-1\right).$$
		 with probability at least $1-2e^{-\lambda}$. Hence, with probability at least $1-2\ell\cdot e^{-\lambda}=1-O(\alpha^{-1}\log(1/\tau)e^{-\lambda})$, this is true simultaneously for every $i\in [\ell]$. We show that in this case $|B_{\ell-1}|\geq n^{1-\tau}$. 
		
		Write $|B_{i}|=(n/2)^{1-\delta}$. Then 
		\begin{align*}
			|B_{i+1}|&\geq |B_{i}|\left(1+\frac{1}{4}\left(\left(\frac{n}{2|B_{i}|}\right)^{\alpha/2}-1\right)\right)= |B_i|\left(1+\frac{1}{4}\left((n/2)^{\delta\alpha/2}-1\right)\right)\\
			& \geq |B_{i}| (n/2)^{\delta\alpha/16}= (n/2)^{1-\delta(1-\alpha/16)}.
		\end{align*}
		In the last inequality, we used that the function $f_{\beta}(x)=\frac{3}{4}+\frac{1}{4}x^{\beta}-x^{\beta/4}$ is nonnegative for $x\geq 1$ and $\beta>0$, which can be seen from $f_{\beta}(1)=0$ and $f'_{\beta}(x)=\frac{\beta}{4}(x^{\beta-1}-x^{\beta/4-1})\geq 0$. By induction, we get $|B_{i}|\geq (n/2)^{1-(1-\alpha/16)^{i}}$ if $|B_{i}|\leq n^{1-\tau}$. Hence, if $(n/2)^{(1-\alpha/16)^{i}}<n^{-\tau}$, then $|B_{i}|\geq n^{1-\tau}$. This does hold if $i\geq 100\alpha^{-1}(\log (1/\tau))$, so $|B_{\ell}|\geq n^{1-\tau}$.
	\end{proof}
	
	In case $\alpha$ is a constant, we can prove a slightly stronger version of the previous lemma, following the same proof. We omit the details.
	
	\begin{lemma}\label{lemma:randompath_v2}
		Let $p,p_c,\tau\in (0,1]$, $\alpha\in (0,\frac{1}{2})$, and let $\ell$ be a positive integer satisfying $\ell\geq C\alpha^{-1}\log(1/\tau)$. Let $n$ be a positive integer which is sufficiently large with respect to $\alpha,\tau,\ell$. Let $G$ be an $\alpha$-maximal graph on $n$ vertices with proper edge coloring $f:E(G)\rightarrow R$, $d:=d(G)$, satisfying $d>(p\cdot p_{c}^{2})^{-1}\cdot n^{\alpha}$.
		
		Let $U$ be a random sample of the vertices, each chosen with probability $p$, and let $Q\subset R$ be a random sample of the colors, each chosen with probability $p_c$. Then with probability at least $1-e^{n^{-\Omega(\alpha^{2})}}$, for every vertex $v\in V(G)$, at least $n^{1-\tau}$ vertices of $G$ can be reached from $v$ by a $(U,Q)$-path of length exactly~$\ell$. 
	\end{lemma}
	
	\subsection{Proofs of the rainbow theorems}
	
	In this section, we present the proofs of Theorems \ref{thm:rainbow_cycle} (rainbow cycle), \ref{thm:rainbow_subdivision} (rainbow subdivision) and \ref{thm:rainbow_robust} (large rainbow subdivision). Note that Theorem \ref{thm:robust} (large subdivision) then immediately follows from Theorem \ref{thm:rainbow_robust}. Let us start with Theorem \ref{thm:rainbow_cycle}, which we restate as follows.
	
	\begin{theorem}
		Let $\epsilon\in (0,1)$. If $n$ is sufficiently large, $G$ is a properly edge colored graph with $n$ vertices and $d(G)\geq (\log n)^{2+\epsilon}$, then $G$ contains a rainbow cycle.
	\end{theorem}
	
	\begin{proof}
		 Let $\alpha=1/\log_2 n$, $p=1$, $p_c=1/4$, $\lambda=(\log n)^{\epsilon/10^6}$. Let $H$ be the subgraph of $G$ maximizing the quantity $d(H)/v(H)^{\alpha}$ and let $m=v(H)$. Then $H$ is $\alpha$-maximal and $d(H)\geq d(G)/2>C\lambda^{7}(\alpha^{2}p_{c}^{2})^{-1}m^{\alpha}$, and  $\lambda\geq C(\log\log m)^{100}\log\frac{2}{p\cdot p_c}$ if $n$ is sufficiently large. Furthermore, let $\tau= 1/\log_3 m$.
		
		Let $v\in V(H)$. Partition $R$ into four parts $Q_1,Q_2,Q_3,Q_4$ randomly such that each color appears independently with probability $p_c=1/4$ in each part. For $i\in [4]$, let $B_{i}$ be the set of vertices that can be reached from $v$ by a $(V(G),Q_{i})$-path. Applying Lemma \ref{lemma:randompath} with the parameters above (with $m$ instead of $n$), we get the following. With probability at least $1-O(\alpha^{-1}\log(1/\tau)e^{-\lambda})>4/5$, we have $|B_{i}|\geq n/3$. Hence, there exists a partition $Q_1,Q_2,Q_3,Q_4$ such that $|B_{i}|\geq n/3$ for every $i\in [4]$. 
		
		Note that there exists $1\leq i<j\leq 4$ such that $B_{i}\cap B_{j}$ is nonempty, let $w\in B_{i}\cap B_{j}$ be an arbitrary vertex. Then $w$ can be reached by a $(V(G),Q_{i})$-path $P_{1}$ and a $(V(G),Q_{j})$-path $P_{2}$ from $v$. The union of $P_1$ and $P_2$ is a rainbow circuit, so in particular it contains a rainbow cycle.
	\end{proof}
	
	In what follows, we prove our results about rainbow subdivisions. In order to do this, we use the result of \cite{AKS03} mentioned in the introduction.
	
	\begin{lemma}\label{lemma:1subdivision}(\cite{AKS03})
		If $G$ is a graph on $n$ vertices with at least $\delta n^{2}$ edges, then $G$ contains a 1-subdivision of $K_{t}$ for $t=\delta n^{1/2}/4$.
	\end{lemma}
	
	Here is Theorem \ref{thm:rainbow_subdivision}, restated.
	
	\begin{theorem}
		Let $t$ be a positive integer, $\epsilon\in (0,1)$. If $n$ is sufficiently large, $G$ is a properly edge colored graph with $n$ vertices and $d(G)\geq (\log n)^{6+\epsilon}$, then $G$ contains a rainbow subdivision of $K_{t}$.
	\end{theorem}
	
	\begin{proof}
		Let $s:=(\log n)^{1+\epsilon/10}$, $\alpha=1/\log_2 n$, $p=p_c=1/s$, and $\lambda=(\log n)^{\epsilon/10^6}$. Let $H$ be the subgraph of $G$ maximizing the quantity $d(H)/v(H)^{\alpha}$ and let $m=v(H)$. Then $H$ is $\alpha$-maximal and $d(H)\geq d(G)/2>C\lambda^7(\alpha^{3}\cdot p\cdot p_{c}^{2})^{-1}m^{\alpha}$, and  $\lambda\geq C(\log\log m)^{100}\log\frac{2}{p\cdot p_c}$ if $n$ is sufficiently large. Furthermore, let $\tau= 1/\log_3 m$.

		Let $v\in V(H)$. Let $U_1,\dots,U_{s}$ be a random partition of $V(H)$ such that each vertex appears in each part independently with probability $p=1/s$. Similarly, let $Q_1,\dots,Q_{s}$ be a random partition of $R$ such that each color appears in each part independently with probability $p_{c}=1/s$. For $i\in [s]$, let $B_{i}$ be the set of vertices which can be reached from $v$ by a $(U_{i},Q_{i})$-path of length  $O(\alpha^{-1}\log (1/\tau))=(\log n)^{1+o(1)}$. Applying Lemma \ref{lemma:randompath} with the above parameters (and $m$ instead of $n$), we conclude that for every $i\in [s]$, with probability at least $1-O(\alpha^{-1}\log(1/\tau)e^{-\lambda})>1-1/2s$, $|B_{i}|\geq n/3$. But then there exist  partitions  $U_1,\dots,U_{s}$ and $Q_1,\dots,Q_s$ such that $|B_{i}|\geq m/3$ for every $i$. Fix such partitions.
		
		Say that a vertex $w\in V(H)$ is good if $w$ is contained in at least $s/6$ of the sets $B_{1},\dots,B_{s}$. Then the number of good vertices is at least $m/6$. Indeed, consider the multiset $B_{1}\cup\dots\cup B_{s}$. It has at least $sm/3$ elements, and vertices that are not good contribute at most $sm/6$ elements. As each good vertex contributes at most $s$ elements, we must have at least $m/6$ good vertices. Note that if $w$ is good, then there are at least $s/6$ internally vertex disjoint paths of length at most $\alpha^{-1}(\log n)^{o(1)}$ between $v$ and $w$ such that no color appears twice on the union of the paths. 
		
		Define the graph $J$ on $V(H)$ in which two vertices $x,y$ are joined by an edge if there are at least $s/6$ internally vertex disjoint paths between $x$ and $y$ such that no color appears twice on the union of the paths. By the preceding argument, every vertex of $J$ has degree at least $n/6$ (note that we do not claim that the same partitions work for every vertex of $J$). Hence, $J$ contains a 1-subdivision of $K_{t}$ by Lemma \ref{lemma:1subdivision}, noting that $t\leq v(J)^{1/2}/100$ if $n$ is sufficiently large. As long as $s\geq t^{2}(\log n)^{1+o(1)}$, we can greedily substitute a rainbow path for each edge of this subdivision, such that these paths are internally vertex disjoint, and no color appears twice. This gives the desired rainbow subdivision of $K_{t}$.
	\end{proof}
	
	Next, we prove Theorem \ref{thm:rainbow_robust}, which we restate as follows.
	
	\begin{theorem}
		There exist constants $c_1,c_2>0$ such that the following holds. Let $\alpha>0$, $\ell>\frac{c_1}{\alpha}\log\frac{1}{\alpha}$ even. Let $G$ be a properly edge colored graph on $n$ vertices with $d(G)\geq n^{\alpha}$. If $n$ is sufficiently large with respect to $\alpha,\ell$, then  $G$ contains a rainbow $(\ell-1)$-subdivision of $K_t$, where $t=n^{c_2\alpha}$.
	\end{theorem}
	
	\begin{proof}
		Let $\alpha'=\alpha/4$, $s=n^{\alpha/6}$, and let $p=p_{c}=1/s$.  Let $H$ be the subgraph of $G$ maximizing the quantity $d(H)/v(H)^{\alpha'}$. Then $H$ is $\alpha'$-maximal and $d(H)\geq d(G)/n^{\alpha'}>n^{3\alpha/4}$. Let $m=v(H)$, then $m\geq d(H)>n^{3\alpha/4}$. Furthermore, let $\tau=\alpha/15$, and assume that $\ell/2\geq C\alpha'^{-1}\log(1/\tau)$, where $C$ is given by Lemma \ref{lemma:randompath_v2}.
		
		Let $U_1,\dots,U_{s}$ be a random partition of $V(G)$ such that each vertex appears in each part independently with probability $p=1/s$. Similarly, let $Q_1,\dots,Q_{s}$ be a random partition of $R$ such that each color appears in each part independently with probability $p_{c}=1/s$. As $d(H)>(p\cdot p_{c}^{2})^{-1}m^{\alpha'}$, we can apply Lemma \ref{lemma:randompath_v2} to conclude the following. For every $i\in [s]$, with probability at least $1-e^{-m^{\Omega(\alpha^{2})}}$, for every $v\in V(G)$ at least $m^{1-\tau}$ of the vertices can be reached from $v$ by a $(U_i,Q_i)$-path of length $\ell/2$. But then there exist partitions $U_1,\dots,U_{s}$ and $Q_1,\dots,Q_s$ such that this is simultaneously true for every $i\in [s]$. Let us fix such partitions. 
		
		Define the edge colored multigraph $K$ on $V(G)$ in which for every $i\in [s]$, we add an edge of color $i$ between $x,y$ if there exists a $(U_i,Q_i)$-path of length $\ell/2$ with endpoints $x$ and $y$. Then $K$ has at least $sm^{2-\tau}$ edges, and there are at most $s$ edges between any pair of vertices. Let $L$ be the graph on $V(G)$ in which $x,y$ are joined by an edge if there are at least $sm^{-\tau}/2\geq n^{\alpha/12}$ edges between $x$ and $y$ in $K$. Note that the nonedges of $L$ contribute at most $sm^{2-\tau}/2$ edges to $K$, and each edge of $L$ contributes at most $s$ edges, so $L$ has at least $m^{2-\tau}/2$ edges. Therefore, by Lemma \ref{lemma:1subdivision}, $L$ contains a 1-subdivision of $K_{t}$ for every  $t\leq m^{1/2-\tau}/8$. Here $m^{1/2-\tau}/8\geq m^{1/4}>n^{\alpha/24}$, so we can choose $t=n^{\alpha/24}$. Let $S$ denote this subdivision. Then $e(S)<t^{2}=n^{\alpha/12}$, so we can assign a color $r_{xy}$ to each edge $xy\in E(S)$ among the colors of the edges connecting $x$ and $y$ in $K$ such that no two edges of $S$ receive the same color. But then choosing a $(U_{r_{xy}},Q_{r_{xy}})$-path $P_{xy}$ for every $xy\in E(S)$, the union of these paths is an $(\ell-1)$-subdivision of $K_{t}$.
	\end{proof}
	
	Finally, let us address the claim we made after Theorem \ref{thm:robust}. As the proof of this is essentially the same as the previous proof, only with some parameters changed, we give only a sketch.
	
	\begin{theorem}\label{thm:robust_formal}
	Let $\alpha>\epsilon>0$, then there exists $\ell_{0}=\ell_{0}(\epsilon,\alpha)$ such that the following holds. If $\ell\geq \ell_0$ is odd and $G$ is a graph on $n$ vertices with $d(G)>n^{\alpha}$, then $G$ contains an $\ell$-subdivision of $K_{t}$ for $t=d(G)^{1/2}n^{-\epsilon}$, if $n$ is sufficiently large.
	\end{theorem}
	
	\begin{proof}[Sketch proof.]
	Repeat the previous proof with the following changes. We only sample vertices, so one can take $p_{c}=1$. Moreover, set $\alpha'=\epsilon/10$, $s=d(G)n^{-\epsilon/10}$, $p=1/s$, $\tau=\alpha'/10$. Then $\ell_0\geq 2C\alpha'^{-1}\log(1/\tau)$ suffices.
	\end{proof}
	
	\section{Cycles in simplicial complexes}\label{sect:simplicial}
	
	\subsection{Simplicial complexes and higher order walks}\label{sect:introductiontosimplices}
	
	In this section, we introduce our notation concerning hypergraphs and simplicial complexes. An $r$-graph $G$ naturally corresponds to the $(r-1)$-dimensional \emph{pure simplicial complex} given by the downward closure of the edge set, that is, the simplicial complex $$S_{G}=S_{G}(0)\cup\dots\cup S_{G}(r),$$
	where $S_{G}(r)=E(G)$ is the edge set of $G$, and $S_{G}(i)$ is the family of $i$ element subsets of the edges for $i=0,\dots,r-1$. We call the elements of $S_{G}(i)$ the \emph{$i$-faces of $S_{G}$}. When talking about edges, we always mean the $r$-faces of $G$. Let $P(G)=S_{G}(r-1)$ denote the set of $(r-1)$-faces, and let $V(G)=S_{G}(1)$ denote the vertices. Also, we set $e(G):=|E(G)|$, $p(G):=|P(G)|$ and $v(G):=|V(G)|$. We say that $H$ is a \emph{subhypergraph} of $G$ if $E(H)\subset E(G)$. 
	
	We will use $r$-graphs and simplicial complexes interchangeably, as in certain situations one is more natural than the other; also, we might identify a family of $r$-elements sets $E$ with the $r$-graph $G$ satisfying $E(G)=E$.
	
	Given $X\subset P(G)$, that is, a set of $(r-1)$-faces, we define its \emph{neighborhood} as $$N_{G}(X)=N(X):=\{f\in P(G)\setminus X:\exists f'\in X, f\cup f'\in E(G)\}.$$ Also, the subhypergraph \emph{induced by $X$} is $$G[X]:=\{e\in E(G):e^{(r-1)}\subset X\},$$ that is, the set of edges whose every $(r-1)$-element subset appear in $X$. Finally, for $f\in S_{G}$, define the \emph{degree} of $f$ as  $$\deg_{G}(f)=\deg(f):=|\{e\in E(G):f\subset e\}|.$$

	In this paper, we are interested in $r$-graphs in which the $(r-1)$-faces have large average degree. Therefore, given an $r$-graph $G$, define its \emph{average degree} as $d(G):=r\cdot e(G)/p(G)$. We will use the following technical lemmas repeatedly.
	
	\begin{lemma}\label{lemma:mindeg}
		Let $G$ be an $r$-graph. Then $G$ contains a subhypergraph $H$ such that $\deg_{H}(f)\geq d(G)/r$ for every $f\in P(H)$.
	\end{lemma}
	
	\begin{proof}
		Let $H$ be a minimal subhypergraph of $G$ with $d(H)\geq d(G)$. Suppose that there exists $f\in P(H)$ such that $\deg_{H}(f)< d(G)/r\leq d(H)/r$. Remove all the edges containing $f$ from $H$, and let $H'$ be the resulting subhypergraph. Then $p(H')\leq p(H)-1$ and $e(H')> e(H)-d(H)/r$, so
		$$d(H')=\frac{r\cdot e(H')}{p(H')}>\frac{r\cdot e(H)-d(H)}{p(H)-1}=d(H),$$
		contradicting the minimality of $H$. Therefore, $H$ suffices.
	\end{proof}
	
	\begin{lemma}\label{lemma:maxdeg}
		Let $G$ be an $r$-graph such that $p(G)=n$ and $\deg(f)\geq d$ for every $f\in P(G)$. Then for every $v\in V(G)$, we have $\deg_{P(G)}(v)\leq rn/d$.
	\end{lemma}
	
	\begin{proof}
		Let $E_v\subset E(G)$ be the set of edges containing $v$, and let $P_v\subset P(G)$ be the set of $(r-1)$-faces containing $v$. Note that $|E_{v}|\leq n$, as each $e\in E_v$ contributes the face $e\setminus\{v\}\in P(G)$, and any two of these faces are different. On the other hand, each $f\in P_v$ appears in at least $d$ edges in $E_v$. Therefore,
		$$|P_v|d\leq |E_v|(r-1)\leq nr,$$
		which gives $\deg_{P(G)}(v)=|P_v|\leq rn/d$.
	\end{proof}
	
	\bigskip
	
	\noindent
	\textbf{Higher order walks, paths, and cycles.} A \emph{higher order walk} is a generalization of a graph walk for  simplicial complexes. Recall that a walk in a graph moves on vertices via edges. This can be extended for simplicial complexes at any level $k\in [r-1]$: a higher order walk traverses $k$-faces via $(k+1)$-faces. Here, we are mostly interested in the case $k=r-1$, so let us define this formally.
	
	\begin{definition}(\emph{Higher order walks})
		A \emph{walk of length $\ell$} in an $r$-graph $G$ is a sequence of $(r-1)$-faces $f_0,\dots,f_{\ell}$ such that $f_{i-1}\cup f_{i}\in E(G)$ for $i\in [\ell]$. Say that the walk is \emph{closed} if $f_0=f_{\ell}$. The simplicial complex \emph{associated with the walk} is the $r$-graph with edges $e_{i}:=f_{i-1}\cup f_{i}$ for $i\in [\ell]$.
	\end{definition}
	
	With slight abuse of notation, when talking about a walk as a topological space, we always mean the simplicial complex associated with it.
	
	We would also like to extend the notion of paths in graphs for simplicial complexes. A natural way to do this is to consider walks with no repeated vertices. In other words, as we traverse the $(r-1)$-faces $f_0,\dots,f_{\ell}$, we uncover one new vertex at each step. This can be encoded in the following short definition.
	
	\begin{definition}(\emph{Higher order paths}) A \emph{path of length $\ell$} in an $r$-graph $G$ is a walk $f_0,\dots,f_{\ell}$ such that $|f_0\cup\dots\cup f_{\ell}|=\ell+r-1$. The \emph{internal vertices} of a path are the elements of $(f_0\cup\dots\cup f_{\ell})\setminus (f_0\cup f_{\ell})$. The \emph{endpoints} of the path are $f_0$ and $f_{\ell}$, and the path is \emph{proper} if its endpoints are disjoint.
	\end{definition}
	
	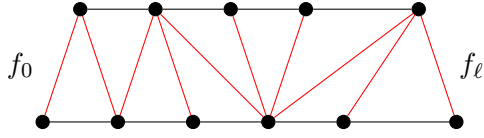
\begin{figure}
		\begin{center}
			\begin{tikzpicture}
				
				\node[vertex,label=below:$1$] (v1) at (-3.5,-0.5) {};
				\node[vertex,label=above:$2$] (v2) at (-3,1) {};
				\node[vertex,label=below:$3$] (v3) at (-2.5,-0.5) {};
				\node[vertex,label=below:$5$] (v5) at (-1.5,-0.5) {};
				\node[vertex,label=above:$4$] (v4) at (-2,1) {};
				\node[vertex,label=above:$7$] (v7) at (-1,1) {};
				\node[vertex,label=below:$6$] (v6) at (-0.5,-0.5) {};
				\node[vertex,label=above:$8$] (v8) at (0,1) {};
				\node[vertex,label=below:$10$] (v10) at (0.5,-0.5) {};
				\node[vertex,label=above:$9$] (v9) at (1.5,1) {};
				\node[vertex,label=below:$11$] (v11) at (2,-0.5) {};
				\draw[red]  (v1) edge (v2); \node at (-3.8,0.25) {$f_0$};
				\draw  (v3) edge (v1);
				\draw[red]  (v2) edge (v3);
				\draw[red]  (v3) edge (v4);
				\draw  (v2) edge (v4);
				\draw[red]  (v4) edge (v5);
				\draw  (v3) edge (v5);
				\draw[red]  (v6) edge (v4);
				\draw  (v5) edge (v6);
				\draw  (v4) edge (v7);
				\draw[red]  (v7) edge (v6);
				\draw[red]  (v6) edge (v8);
				\draw[red]  (v6) edge (v9);
				\draw  (v7) edge (v8);
				\draw  (v8) edge (v9);
				\draw  (v6) edge (v10);
				\draw[red]  (v10) edge (v9);
				\draw[red]  (v9) edge (v11);  \node at (2.2,0.25) {$f_{\ell}$};
				\draw  (v11) edge (v10);
			\end{tikzpicture}
		\end{center}
		\caption{An illustration of a path in a 3-graph, where the numbers denote the order of the vertices, and the red edges denote $f_0,\dots,f_{\ell}$.}
		\label{fig:paths}
	\end{figure}
	
	Indeed, as $|f_{i}\setminus f_{i-1}|=1$ for $i\in [\ell]$ in a walk $f_0,\dots,f_{\ell}$, we have $|f_0\cup\dots\cup f_{\ell}|\leq \ell+r-1$, with equality if and only if $f_{i}\setminus f_{i-1}$ is disjoint from $f_0\cup\dots\cup f_{i-1}$ for every $i$. 
 
 In case the path is proper, we can also define the \emph{order of the vertices}. First, order the vertices of $f_0$ according to the order they disappear, that is, $v$ is before $w$ if the largest index $i$ such that $v\in f_i$ is smaller than the largest index $j$ such that $w\in f_j$. Next, order the vertices not in $f_0$ according to the order they appear, that is, $v$ is before $w$ if the smallest index $i$ such that $v\in f_i$ is smaller than the smallest index $j$ such that $w\in f_j$. Finally, every vertex in $f_0$ is before every vertex not in $f_0$. This defines a total ordering of the vertices in case the path is proper, see Figure \ref{fig:paths}.
 
 From a topological perspective, if $f_0,\dots,f_{\ell}$ is a path, then the pure simplicial complex associated with it is homeomorphic to the ball $B^{r-1}$. Finally, let us define cycles.
	
	\begin{definition}(\emph{Higher order cycles})
		A \emph{cycle of length $\ell$} in an $r$-graph $G$ is a closed walk $f_0,\dots,f_{\ell}$ such that $|f_0\cup\dots\cup f_{\ell}|=\ell$, and there exist $1\leq i<j\leq \ell$ such that $f_i\cap f_j=\emptyset$.
	\end{definition}

	Let us accompany this definition with some explanation. If $f_0,\dots,f_{\ell}$ is a cycle and $f_i\cap f_j=\emptyset$ for some $i<j$, then the walks $P=(f_i,f_{i+1},\dots,f_j)$ and $P'=(f_j,f_{j+1},\dots,f_i)$ are both proper paths (where indices are meant modulo $\ell$). To see this, note that $v(P)\leq j-i+r-1$, and $v(P')\leq \ell-j+i+r-1$, so 
	\begin{equation}\label{equ:paths}
		|f_{0}\cup\dots\cup f_{\ell}|\leq v(P)+v(P')-|f_i|-|f_j|\leq \ell.
	\end{equation}
	Hence, as the left and right hand sides are equal, we must have $v(P)=j-i+r-1$, and $v(P')=\ell-j+i+r-1$, so $P$ and $P'$ are paths. The inequality (\ref{equ:paths}) also shows that $P$ and $P'$ do not share any internal vertices. 
	
	Furthermore, consider the simplicial complex $S$ associated with the cycle $f_0,\dots,f_{\ell}$ from a topological perspective. In the case $r=2$, we get the usual notion of cycles, which are always homeomorphic to $S^1$. However, in dimension 2, $S$ is homeomorphic to either the \emph{cylinder} $S^1\times B^1$, or the \emph{M\"obius strip} $M$, see Figure \ref{fig:simplicialcycle} for an illustration. In general, for $r\geq 3$, we always have two cases.

	\begin{lemma}
		For $r\geq 3$, a cycle in an $r$-graph is homeomorphic to either $S^1\times B^{r-2}$ or $M\times B^{r-3}$.
	\end{lemma}
	
	Roughly, this is true as we get a cycle by taking a proper path $f_0,\dots,f_{\ell}$, and identifying $f_0$ and $f_{\ell}$. Depending on the identification, we get either an orientable simplicial complex, which is then $S^1\times B^{r-2}$, or a non-orientable one, which is then $M\times B^{r-3}$.
	
	We remark that in the case $r=3$, our notion of cycle coincides with the notion of topological cycles introduced in \cite{KPTZ}. In the same paper, it is proved that if a cycle in a 3-graph is also 3-partite, then the parity of its length determines whether it is a cylinder or a M\"obius strip.
	
	\begin{lemma}\label{lemma:characterization}(\cite{KPTZ})
		Let $C$ be a cycle of length $\ell$ in a 3-graph, which is also 3-partite. Then $C$ is homeomorphic to the cylinder if $\ell$ is even, and $C$ is homeomorphic to the M\"obius strip if $\ell$ is odd.
	\end{lemma}
	
	\bigskip
	
	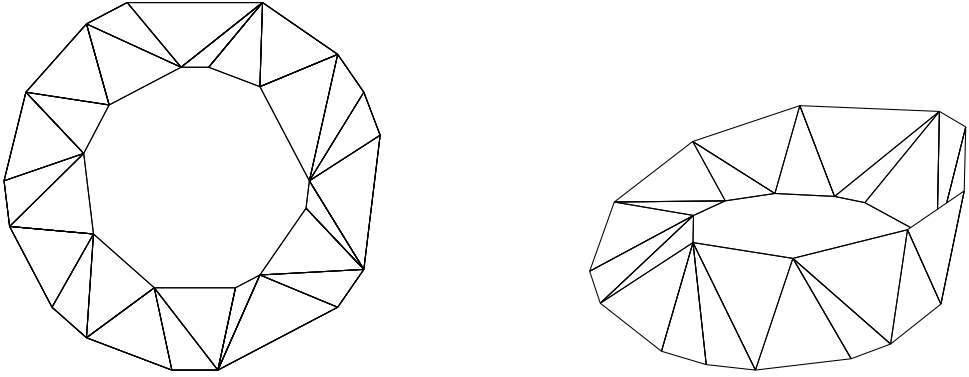
\begin{figure}
		\begin{minipage}[b]{0.45\textwidth}
			\begin{center}
				\begin{tikzpicture}
					\draw (4.05897,2.57318) -- (5,3.18029) ;
\draw (5,3.18029) -- (4.78313,3.75212) ;
\draw (4.05897,2.57318) -- (4.78313,3.75212) ;
\draw (4.78313,3.75212) -- (5,3.18029) ;
\draw (4.05897,2.57318) -- (5,3.18029) ;
\draw (5,3.18029) -- (4.78313,3.75212) ;
\draw (4.05897,2.57318) -- (4.78313,3.75212) ;
\draw (4.78313,3.75212) -- (5,3.18029) ;
\draw (4.05897,2.57318) -- (4.78313,3.75212) ;
\draw (4.78313,3.75212) -- (4.43572,4.25543) ;
\draw (4.05897,2.57318) -- (4.43572,4.25543) ;
\draw (4.43572,4.25543) -- (4.78313,3.75212) ;
\draw (4.05897,2.57318) -- (4.78313,3.75212) ;
\draw (4.78313,3.75212) -- (4.43572,4.25543) ;
\draw (4.05897,2.57318) -- (4.43572,4.25543) ;
\draw (4.43572,4.25543) -- (4.78313,3.75212) ;
\draw (4.05897,2.57318) -- (4.43572,4.25543) ;
\draw (4.43572,4.25543) -- (3.40152,3.82585) ;
\draw (4.05897,2.57318) -- (3.40152,3.82585) ;
\draw (3.40152,3.82585) -- (4.43572,4.25543) ;
\draw (4.05897,2.57318) -- (4.43572,4.25543) ;
\draw (4.43572,4.25543) -- (3.40152,3.82585) ;
\draw (4.05897,2.57318) -- (3.40152,3.82585) ;
\draw (3.40152,3.82585) -- (4.43572,4.25543) ;
\draw (3.40152,3.82585) -- (4.43572,4.25543) ;
\draw (4.43572,4.25543) -- (3.43644,4.94518) ;
\draw (3.40152,3.82585) -- (3.43644,4.94518) ;
\draw (3.43644,4.94518) -- (4.43572,4.25543) ;
\draw (3.40152,3.82585) -- (4.43572,4.25543) ;
\draw (4.43572,4.25543) -- (3.43644,4.94518) ;
\draw (3.40152,3.82585) -- (3.43644,4.94518) ;
\draw (3.43644,4.94518) -- (4.43572,4.25543) ;
\draw (3.40152,3.82585) -- (3.43644,4.94518) ;
\draw (3.43644,4.94518) -- (2.72033,4.0842) ;
\draw (3.40152,3.82585) -- (2.72033,4.0842) ;
\draw (2.72033,4.0842) -- (3.43644,4.94518) ;
\draw (3.40152,3.82585) -- (3.43644,4.94518) ;
\draw (3.43644,4.94518) -- (2.72033,4.0842) ;
\draw (3.40152,3.82585) -- (2.72033,4.0842) ;
\draw (2.72033,4.0842) -- (3.43644,4.94518) ;
\draw (2.72033,4.0842) -- (3.43644,4.94518) ;
\draw (3.43644,4.94518) -- (2.35339,4.0842) ;
\draw (2.72033,4.0842) -- (2.35339,4.0842) ;
\draw (2.35339,4.0842) -- (3.43644,4.94518) ;
\draw (2.72033,4.0842) -- (3.43644,4.94518) ;
\draw (3.43644,4.94518) -- (2.35339,4.0842) ;
\draw (2.72033,4.0842) -- (2.35339,4.0842) ;
\draw (2.35339,4.0842) -- (3.43644,4.94518) ;
\draw (2.35339,4.0842) -- (3.43644,4.94518) ;
\draw (3.43644,4.94518) -- (1.63728,4.94518) ;
\draw (2.35339,4.0842) -- (1.63728,4.94518) ;
\draw (1.63728,4.94518) -- (3.43644,4.94518) ;
\draw (2.35339,4.0842) -- (3.43644,4.94518) ;
\draw (3.43644,4.94518) -- (1.63728,4.94518) ;
\draw (2.35339,4.0842) -- (1.63728,4.94518) ;
\draw (1.63728,4.94518) -- (3.43644,4.94518) ;
\draw (2.35339,4.0842) -- (1.63728,4.94518) ;
\draw (1.63728,4.94518) -- (1.09576,4.66097) ;
\draw (2.35339,4.0842) -- (1.09576,4.66097) ;
\draw (1.09576,4.66097) -- (1.63728,4.94518) ;
\draw (2.35339,4.0842) -- (1.63728,4.94518) ;
\draw (1.63728,4.94518) -- (1.09576,4.66097) ;
\draw (2.35339,4.0842) -- (1.09576,4.66097) ;
\draw (1.09576,4.66097) -- (1.63728,4.94518) ;
\draw (2.35339,4.0842) -- (1.09576,4.66097) ;
\draw (1.09576,4.66097) -- (1.39754,3.58253) ;
\draw (2.35339,4.0842) -- (1.39754,3.58253) ;
\draw (1.39754,3.58253) -- (1.09576,4.66097) ;
\draw (2.35339,4.0842) -- (1.09576,4.66097) ;
\draw (1.09576,4.66097) -- (1.39754,3.58253) ;
\draw (2.35339,4.0842) -- (1.39754,3.58253) ;
\draw (1.39754,3.58253) -- (1.09576,4.66097) ;
\draw (1.39754,3.58253) -- (1.09576,4.66097) ;
\draw (1.09576,4.66097) -- (0.290582,3.75212) ;
\draw (1.39754,3.58253) -- (0.290582,3.75212) ;
\draw (0.290582,3.75212) -- (1.09576,4.66097) ;
\draw (1.39754,3.58253) -- (1.09576,4.66097) ;
\draw (1.09576,4.66097) -- (0.290582,3.75212) ;
\draw (1.39754,3.58253) -- (0.290582,3.75212) ;
\draw (0.290582,3.75212) -- (1.09576,4.66097) ;
\draw (1.39754,3.58253) -- (0.290582,3.75212) ;
\draw (0.290582,3.75212) -- (1.05897,2.93745) ;
\draw (1.39754,3.58253) -- (1.05897,2.93745) ;
\draw (1.05897,2.93745) -- (0.290582,3.75212) ;
\draw (1.39754,3.58253) -- (0.290582,3.75212) ;
\draw (0.290582,3.75212) -- (1.05897,2.93745) ;
\draw (1.39754,3.58253) -- (1.05897,2.93745) ;
\draw (1.05897,2.93745) -- (0.290582,3.75212) ;
\draw (1.05897,2.93745) -- (0.290582,3.75212) ;
\draw (0.290582,3.75212) -- (0,2.57318) ;
\draw (1.05897,2.93745) -- (0,2.57318) ;
\draw (0,2.57318) -- (0.290582,3.75212) ;
\draw (1.05897,2.93745) -- (0.290582,3.75212) ;
\draw (0.290582,3.75212) -- (0,2.57318) ;
\draw (1.05897,2.93745) -- (0,2.57318) ;
\draw (0,2.57318) -- (0.290582,3.75212) ;
\draw (1.05897,2.93745) -- (0,2.57318) ;
\draw (0,2.57318) -- (0.0737159,1.96607) ;
\draw (1.05897,2.93745) -- (0.0737159,1.96607) ;
\draw (0.0737159,1.96607) -- (0,2.57318) ;
\draw (1.05897,2.93745) -- (0,2.57318) ;
\draw (0,2.57318) -- (0.0737159,1.96607) ;
\draw (1.05897,2.93745) -- (0.0737159,1.96607) ;
\draw (0.0737159,1.96607) -- (0,2.57318) ;
\draw (1.05897,2.93745) -- (0.0737159,1.96607) ;
\draw (0.0737159,1.96607) -- (1.18909,1.86582) ;
\draw (1.05897,2.93745) -- (1.18909,1.86582) ;
\draw (1.18909,1.86582) -- (0.0737159,1.96607) ;
\draw (1.05897,2.93745) -- (0.0737159,1.96607) ;
\draw (0.0737159,1.96607) -- (1.18909,1.86582) ;
\draw (1.05897,2.93745) -- (1.18909,1.86582) ;
\draw (1.18909,1.86582) -- (0.0737159,1.96607) ;
\draw (1.18909,1.86582) -- (0.0737159,1.96607) ;
\draw (0.0737159,1.96607) -- (0.637991,0.890932) ;
\draw (1.18909,1.86582) -- (0.637991,0.890932) ;
\draw (0.637991,0.890932) -- (0.0737159,1.96607) ;
\draw (1.18909,1.86582) -- (0.0737159,1.96607) ;
\draw (0.0737159,1.96607) -- (0.637991,0.890932) ;
\draw (1.18909,1.86582) -- (0.637991,0.890932) ;
\draw (0.637991,0.890932) -- (0.0737159,1.96607) ;
\draw (1.18909,1.86582) -- (0.637991,0.890932) ;
\draw (0.637991,0.890932) -- (1.09576,0.485387) ;
\draw (1.18909,1.86582) -- (1.09576,0.485387) ;
\draw (1.09576,0.485387) -- (0.637991,0.890932) ;
\draw (1.18909,1.86582) -- (0.637991,0.890932) ;
\draw (0.637991,0.890932) -- (1.09576,0.485387) ;
\draw (1.18909,1.86582) -- (1.09576,0.485387) ;
\draw (1.09576,0.485387) -- (0.637991,0.890932) ;
\draw (1.18909,1.86582) -- (1.09576,0.485387) ;
\draw (1.09576,0.485387) -- (1.99711,1.14998) ;
\draw (1.18909,1.86582) -- (1.99711,1.14998) ;
\draw (1.99711,1.14998) -- (1.09576,0.485387) ;
\draw (1.18909,1.86582) -- (1.09576,0.485387) ;
\draw (1.09576,0.485387) -- (1.99711,1.14998) ;
\draw (1.18909,1.86582) -- (1.99711,1.14998) ;
\draw (1.99711,1.14998) -- (1.09576,0.485387) ;
\draw (1.99711,1.14998) -- (1.09576,0.485387) ;
\draw (1.09576,0.485387) -- (2.23107,0.0548177) ;
\draw (1.99711,1.14998) -- (2.23107,0.0548177) ;
\draw (2.23107,0.0548177) -- (1.09576,0.485387) ;
\draw (1.99711,1.14998) -- (1.09576,0.485387) ;
\draw (1.09576,0.485387) -- (2.23107,0.0548177) ;
\draw (1.99711,1.14998) -- (2.23107,0.0548177) ;
\draw (2.23107,0.0548177) -- (1.09576,0.485387) ;
\draw (1.99711,1.14998) -- (2.23107,0.0548177) ;
\draw (2.23107,0.0548177) -- (2.84264,0.054817) ;
\draw (1.99711,1.14998) -- (2.84264,0.054817) ;
\draw (2.84264,0.054817) -- (2.23107,0.0548177) ;
\draw (1.99711,1.14998) -- (2.23107,0.0548177) ;
\draw (2.23107,0.0548177) -- (2.84264,0.054817) ;
\draw (1.99711,1.14998) -- (2.84264,0.054817) ;
\draw (2.84264,0.054817) -- (2.23107,0.0548177) ;
\draw (1.99711,1.14998) -- (2.84264,0.054817) ;
\draw (2.84264,0.054817) -- (3.07661,1.14998) ;
\draw (1.99711,1.14998) -- (3.07661,1.14998) ;
\draw (3.07661,1.14998) -- (2.84264,0.054817) ;
\draw (1.99711,1.14998) -- (2.84264,0.054817) ;
\draw (2.84264,0.054817) -- (3.07661,1.14998) ;
\draw (1.99711,1.14998) -- (3.07661,1.14998) ;
\draw (3.07661,1.14998) -- (2.84264,0.054817) ;
\draw (3.07661,1.14998) -- (2.84264,0.054817) ;
\draw (2.84264,0.054817) -- (3.40152,1.3205) ;
\draw (3.07661,1.14998) -- (3.40152,1.3205) ;
\draw (3.40152,1.3205) -- (2.84264,0.054817) ;
\draw (3.07661,1.14998) -- (2.84264,0.054817) ;
\draw (2.84264,0.054817) -- (3.40152,1.3205) ;
\draw (3.07661,1.14998) -- (3.40152,1.3205) ;
\draw (3.40152,1.3205) -- (2.84264,0.054817) ;
\draw (3.40152,1.3205) -- (2.84264,0.054817) ;
\draw (2.84264,0.054817) -- (4.43572,0.890929) ;
\draw (3.40152,1.3205) -- (4.43572,0.890929) ;
\draw (4.43572,0.890929) -- (2.84264,0.054817) ;
\draw (3.40152,1.3205) -- (2.84264,0.054817) ;
\draw (2.84264,0.054817) -- (4.43572,0.890929) ;
\draw (3.40152,1.3205) -- (4.43572,0.890929) ;
\draw (4.43572,0.890929) -- (2.84264,0.054817) ;
\draw (3.40152,1.3205) -- (4.43572,0.890929) ;
\draw (4.43572,0.890929) -- (4.78313,1.39424) ;
\draw (3.40152,1.3205) -- (4.78313,1.39424) ;
\draw (4.78313,1.39424) -- (4.43572,0.890929) ;
\draw (3.40152,1.3205) -- (4.43572,0.890929) ;
\draw (4.43572,0.890929) -- (4.78313,1.39424) ;
\draw (3.40152,1.3205) -- (4.78313,1.39424) ;
\draw (4.78313,1.39424) -- (4.43572,0.890929) ;
\draw (3.40152,1.3205) -- (4.78313,1.39424) ;
\draw (4.78313,1.39424) -- (4.01474,2.20891) ;
\draw (3.40152,1.3205) -- (4.01474,2.20891) ;
\draw (4.01474,2.20891) -- (4.78313,1.39424) ;
\draw (3.40152,1.3205) -- (4.78313,1.39424) ;
\draw (4.78313,1.39424) -- (4.01474,2.20891) ;
\draw (3.40152,1.3205) -- (4.01474,2.20891) ;
\draw (4.01474,2.20891) -- (4.78313,1.39424) ;
\draw (4.01474,2.20891) -- (4.78313,1.39424) ;
\draw (4.78313,1.39424) -- (4.05897,2.57318) ;
\draw (4.01474,2.20891) -- (4.05897,2.57318) ;
\draw (4.05897,2.57318) -- (4.78313,1.39424) ;
\draw (4.01474,2.20891) -- (4.78313,1.39424) ;
\draw (4.78313,1.39424) -- (4.05897,2.57318) ;
\draw (4.01474,2.20891) -- (4.05897,2.57318) ;
\draw (4.05897,2.57318) -- (4.78313,1.39424) ;
\draw (4.05897,2.57318) -- (4.78313,1.39424) ;
\draw (4.78313,1.39424) -- (5,3.18029) ;
\draw (4.05897,2.57318) -- (5,3.18029) ;
\draw (5,3.18029) -- (4.78313,1.39424) ;
\draw (4.05897,2.57318) -- (4.78313,1.39424) ;
\draw (4.78313,1.39424) -- (5,3.18029) ;
\draw (4.05897,2.57318) -- (5,3.18029) ;
\draw (5,3.18029) -- (4.78313,1.39424) ;
				\end{tikzpicture}
			\end{center}
		\end{minipage}
		\begin{minipage}[b]{0.45\textwidth}
			\begin{center}
				\begin{tikzpicture}
					\draw (4.22576,2.60779) -- (4.66951,1.61858) ;
\draw (4.66951,1.61858) -- (4.97948,3.1245) ;
\draw (4.22576,2.60779) -- (4.97948,3.1245) ;
\draw (4.97948,3.1245) -- (4.66951,1.61858) ;
\draw (4.97948,3.1245) -- (4.66951,1.61858) ;
\draw (4.74651,2.96479) -- (5,3.97631) ;
\draw (4.97948,3.1245) -- (5,3.97631) ;
\draw (5,3.97631) -- (4.74651,2.96479) ;
\draw (5,3.97631) -- (4.74651,2.96479) ;
\draw (4.62496,2.88146) -- (4.65003,4.18221) ;
\draw (5,3.97631) -- (4.65003,4.18221) ;
\draw (4.65003,4.18221) -- (4.62496,2.88146) ;
\draw (4.65003,4.18221) -- (4.62496,2.88146) ;
\draw (4.26772,2.63656) -- (3.65879,2.96866) ;
\draw (4.65003,4.18221) -- (3.65879,2.96866) ;
\draw (3.65879,2.96866) -- (4.26772,2.63656) ;
\draw (4.65003,4.18221) -- (3.65879,2.96866) ;
\draw (3.65879,2.96866) -- (3.25561,3.0512) ;
\draw (4.65003,4.18221) -- (3.25561,3.0512) ;
\draw (3.25561,3.0512) -- (3.65879,2.96866) ;
\draw (4.65003,4.18221) -- (3.25561,3.0512) ;
\draw (3.25561,3.0512) -- (2.79422,4.25902) ;
\draw (4.65003,4.18221) -- (2.79422,4.25902) ;
\draw (2.79422,4.25902) -- (3.25561,3.0512) ;
\draw (2.79422,4.25902) -- (3.25561,3.0512) ;
\draw (3.25561,3.0512) -- (2.46295,3.09055) ;
\draw (2.79422,4.25902) -- (2.46295,3.09055) ;
\draw (2.46295,3.09055) -- (3.25561,3.0512) ;
\draw (2.79422,4.25902) -- (2.46295,3.09055) ;
\draw (2.46295,3.09055) -- (1.37056,3.78242) ;
\draw (2.79422,4.25902) -- (1.37056,3.78242) ;
\draw (1.37056,3.78242) -- (2.46295,3.09055) ;
\draw (1.37056,3.78242) -- (2.46295,3.09055) ;
\draw (2.46295,3.09055) -- (1.80303,2.99252) ;
\draw (1.37056,3.78242) -- (1.80303,2.99252) ;
\draw (1.80303,2.99252) -- (2.46295,3.09055) ;
\draw (1.37056,3.78242) -- (1.80303,2.99252) ;
\draw (1.80303,2.99252) -- (0.327797,2.97707) ;
\draw (1.37056,3.78242) -- (0.327797,2.97707) ;
\draw (0.327797,2.97707) -- (1.80303,2.99252) ;
\draw (0.327797,2.97707) -- (1.80303,2.99252) ;
\draw (1.80303,2.99252) -- (1.38061,2.79807) ;
\draw (0.327797,2.97707) -- (1.38061,2.79807) ;
\draw (1.38061,2.79807) -- (1.80303,2.99252) ;
\draw (0.327797,2.97707) -- (1.38061,2.79807) ;
\draw (1.38061,2.79807) -- (0,2.05559) ;
\draw (0.327797,2.97707) -- (0,2.05559) ;
\draw (0,2.05559) -- (1.38061,2.79807) ;
\draw (0,2.05559) -- (1.38061,2.79807) ;
\draw (1.38061,2.79807) -- (0.141517,1.63068) ;
\draw (0,2.05559) -- (0.141517,1.63068) ;
\draw (0.141517,1.63068) -- (1.38061,2.79807) ;
\draw (0.141517,1.63068) -- (1.38061,2.79807) ;
\draw (1.38061,2.79807) -- (1.37393,2.43721) ;
\draw (0.141517,1.63068) -- (1.37393,2.43721) ;
\draw (1.37393,2.43721) -- (1.38061,2.79807) ;
\draw (0.141517,1.63068) -- (1.37393,2.43721) ;
\draw (1.37393,2.43721) -- (0.954995,0.991333) ;
\draw (0.141517,1.63068) -- (0.954995,0.991333) ;
\draw (0.954995,0.991333) -- (1.37393,2.43721) ;
\draw (0.954995,0.991333) -- (1.37393,2.43721) ;
\draw (1.37393,2.43721) -- (1.55008,0.813581) ;
\draw (0.954995,0.991333) -- (1.55008,0.813581) ;
\draw (1.55008,0.813581) -- (1.37393,2.43721) ;
\draw (1.55008,0.813581) -- (1.37393,2.43721) ;
\draw (1.37393,2.43721) -- (2.20404,0.74098) ;
\draw (1.55008,0.813581) -- (2.20404,0.74098) ;
\draw (2.20404,0.74098) -- (1.37393,2.43721) ;
\draw (2.20404,0.74098) -- (1.37393,2.43721) ;
\draw (1.37393,2.43721) -- (2.70226,2.22702) ;
\draw (2.20404,0.74098) -- (2.70226,2.22702) ;
\draw (2.70226,2.22702) -- (1.37393,2.43721) ;
\draw (2.20404,0.74098) -- (2.70226,2.22702) ;
\draw (2.70226,2.22702) -- (3.47646,0.891687) ;
\draw (2.20404,0.74098) -- (3.47646,0.891687) ;
\draw (3.47646,0.891687) -- (2.70226,2.22702) ;
\draw (3.47646,0.891687) -- (2.70226,2.22702) ;
\draw (2.70226,2.22702) -- (4.00126,1.08751) ;
\draw (3.47646,0.891687) -- (4.00126,1.08751) ;
\draw (4.00126,1.08751) -- (2.70226,2.22702) ;
\draw (4.00126,1.08751) -- (2.70226,2.22702) ;
\draw (2.70226,2.22702) -- (4.22576,2.60779) ;
\draw (4.00126,1.08751) -- (4.22576,2.60779) ;
\draw (4.22576,2.60779) -- (2.70226,2.22702) ;
\draw (4.00126,1.08751) -- (4.22576,2.60779) ;
\draw (4.22576,2.60779) -- (4.66951,1.61858) ;
\draw (4.00126,1.08751) -- (4.66951,1.61858) ;
\draw (4.66951,1.61858) -- (4.22576,2.60779) ;
				\end{tikzpicture}
			\end{center}
		\end{minipage}
		\caption{Two cycles in 3-graphs. The left is homeomorphic to the cylinder, while the right is homeomorphic to the M\"obius strip.}
		\label{fig:simplicialcycle}
	\end{figure}
	
	\noindent
	\textbf{Finding long higher order paths.} A simple result in extremal graph theory states that if $G$ is a graph with minimum degree $d$, then $G$ contains a path of length $d$. This can be easily generalized to show that if $G$ is an $r$-graph such that $p(G)=n$ and $\deg(f)\geq d$ for every $f\in P(G)$, then $G$ contains a \emph{tight path} of length $d$. Here, a tight path of length $d$ is a sequence of vertices $x_1,\dots,x_{d+r-1}$ such that any $r$ consecutive vertices form an edge. Clearly, setting $f_i=\{x_{i+1},\dots,x_{i+r-1}\}$ for $i\in \{0,\dots,d\}$, the sequence $f_0,\dots,f_{d}$ is a path of length $d$. In what comes, we present a variation of this result for $r\geq 3$ in which we restrict the endpoints of the path. This result will be used later.
	
	\begin{lemma}\label{lemma:longhyppaths}
		Let $r,\ell,d$ be positive integers such that $r\geq 3$ and $\ell>r$, and let $G$ be an $r$-graph such that $\deg(f)\geq d$ for every $f\in P(G)$. Let $F\subset P(G)$ such that $|F|\geq 2r\ell\cdot p(G)/d.$  Then $G$ contains a proper path of length $\ell$, whose both endpoints are in $F$.
	\end{lemma}
	
	\begin{proof}
		Let $E_0\subset E(G)$ be the set of edges $e$ which have an $(r-1)$-face contained in $F$. As $\deg(f)\geq d$ for every $f\in P(G)$, we have $|E_0|\geq |F|d/r$. Let $E_1\subset E_0$ be the set of edges with  exactly one face contained in $F$, and let $E_2=E_0\setminus E_1$. Let $j\in \{1,2\}$ such that $|E_{j}|\geq |E_{3-j}|$. Then $d(E_j)\geq |F|d/2p(G)\geq r\ell$, so by Lemma \ref{lemma:mindeg}, $E_j$ contains a subhypergraph $H$ such that for every $f\in P(H)$, $\deg_{H}(f)\geq \ell$. Consider two cases.
		\begin{description}
			
			\item[Case 1.] $j=1$.
			
			Let $\{v_1,\dots,v_{r-1}\}\in F\cap P(H)$. Then $H$ contains a tight path of length $\ell-1$ with vertices $v_1,\dots,v_{\ell+r-2}$, as we can greedily choose the vertex $v_{i}$ for $i=r,\dots,\ell+r-2$ such that $e_{i-r+1}:=\{v_{i-r+1},\dots,v_{i}\}\in E(H)$ and $v_{i}\not\in\{v_1,\dots,v_{i-r}\}$. As $r\geq 3$, and $e_{\ell-1}$ has exactly one face in $F$, there exists a face $f_{\ell-1}\subset e_{\ell-1}$ such that $f_{\ell-1}\neq f_{\ell-2}$, and $f_{\ell-1}\not\in F$. Let $w\in V(H)$ such that $e_{\ell}=f_{\ell-1}\cup \{w\}$ is an edge and $w\not\in \{v_1,\dots,v_{\ell-1}\}$. This exists as $\deg_{H}(f_{\ell-1})\geq \ell$. Let $f_{\ell}$ be the unique face of $e_{\ell}$ contained in $F$, then $f_{\ell}\neq f_{\ell-1}$. Note that $f_0,\dots,f_{\ell}$ is a path, whose endpoints are in $F$. It is also proper: $\ell>r$, so $f_0\cap f_{\ell}=\emptyset$
			
			\item[Case 2.] $j=2$.
			
			$H$ contains a tight path of length $\ell$ with vertices $v_1,\dots,v_{\ell+r-1}$, as we can select any $\{v_1,\dots,v_{r-1}\}\in P(H)$, and then greedily choose the vertex $v_{i}$ for $i=r,\dots,\ell+r-1$ such that $e_{i-r+1}:=\{v_{i-r+1},\dots,v_{i}\}\in E(H)$ and $v_{i}\not\in\{v_1,\dots,v_{i-r}\}$. For $i=1,\dots,\ell-1$, let $f_i=e_{i}\cap e_{i+1}$. As $e_1$ contains at least two faces in $F$, there exists $f_0\in F$ such that $f_0\neq f_1$ and $f_0\subset e_1$. Similarly, there exists $f_{\ell}\in F$ such that $f_{\ell}\neq f_{\ell-1}$ and $f_{\ell}\in e_{\ell}$. But then $f_0,\dots,f_{\ell}$ is a path, whose endpoints are in $F$. It is also proper: $\ell> r$, so $f_0\cap f_{\ell}=\emptyset$.
			
		\end{description}
		
	\end{proof}
	
	Interestingly, this result has no analogue for graphs in case $\ell$ is odd. Indeed, if $G$ is a bipartite graph, and $F$ is one of its vertex classes, then $G$ does not contain a path of length $\ell$, whose endpoints are in $F$.
	
	\bigskip
	
	\noindent	
	\textbf{$3$-graphs without short cycles.}  Finally, we present an argument showing that there are 3-graphs with $n$ vertices and more than $n^{2+\alpha}$ edges avoiding small homeomorphic copies of the cylinder and the M\"obius strip, as promised in the Introduction.
	
	\begin{lemma}\label{lemma:construction}
		Let $\alpha>0$. For every sufficiently large $n$, there exists a 3-graph $G$ with $n$ vertices and more than $n^{2+\alpha}$ edges such that $G$ avoids homeomorphic copies of the cylinder and the M\"obius strip on at most $1/\alpha$ vertices.
	\end{lemma}
	
	\begin{proof}
		Let $p=12n^{-1+\alpha}$, and consider the $3$-graph $G$ on $n$ vertices, in which each of the $\binom{n}{3}$ triples is present as an edge independently, with probability $p$. Let $N_{k}$ denote the number of homeomorphic copies of the cylinder and M\"obius strip on $k$ vertices in $G$, and let $N=\sum_{k\leq 1/\alpha}N_{k}$
		
		Suppose that $C$ is a homeomorphic copy of the cylinder or the M\"obius strip. Then the Euler characteristic of $C$ is $0$, which means that $e(C)-p(C)+v(C)=0$. Let $a$ be the number of vertices on the boundary of $C$, and $b$ be the number of $2$-faces on the boundary. As the boundary is homeomorphic to either $S^1\cup S^1$ or $S^1$, we have $a=b$. Furthermore, counting 2-faces by edges, we have $3e(C)=2p(C)-b$. By this, we get $e(C)-(3e(C)+b)/2+v(C)=0$, that is, $v(C)-e(C)/2-b/2=0$. As $v(C)\geq a=b$, this implies that $e(C)\geq v(C)$.
		
		Let $h_{k}$ denote the total number of homeomorphic copies of the cylinder and the M\"obius strip on a fixed $k$ element vertex set. Then by earlier discussion, $N_{k}\leq h_{k}\binom{n}{k}p^{k}<h_k n^{k}p^{k}$. Therefore,
		$$\mathbb{E}(e(G)-N)\geq \binom{n}{3}p-\sum_{k\leq 1/\alpha}h_k n^{k}p^{k}> n^{2+\alpha},$$
		assuming $n$ is sufficiently large with respect to $\alpha$. Hence, there exists $G$ satisfying $e(G)-N>n^{2+\alpha}$. Fixing such a $G$ and deleting an edge from every homeomorphic copy of the cylinder and M\"obius strip on at most $1/\alpha$ vertices, we get a 3-graph with the desired properties.
	\end{proof}

	\subsection{\texorpdfstring{$\alpha$}{a}-maximal simplicial complexes}\label{sect:alphasimplicial}
	
	In this section, we extend the notion of $\alpha$-maximality to simplicial complexes, and study its properties. Here and later, we only focus on the case when $\alpha$ is viewed as a constant, unlike in the graph setting.
	
	\begin{definition}(\emph{$\alpha$-maximal simplicial complex})
		For $\alpha\in (0,1)$, the $r$-graph $G$ is \emph{$\alpha$-maximal}, if $d(G)/p(G)^{\alpha}\geq d(H)/p(H)^{\alpha}$ for every subhypergraph $H$ of $G$.
	\end{definition}
	
	By a well known lemma of Lov\'asz \cite{Lovasz93}, which is a relaxation of the celebrated Kruskal-Katona theorem \cite{Katona68,Kruskal63}, we have the following inequality between $e(G)$ and $d(G)$: if $e(G)=\binom{x}{r}$ for some real number $x\geq r$, then $p(G)\geq \binom{x}{r-1}$, so $d(G)\leq x-r+1$. Therefore, if $\alpha>1/(r-1)$, $d(G)/p(G)^{\alpha}\rightarrow 0$ as $p(G)\rightarrow \infty$, showing that there are only finitely many $\alpha$-maximal $r$-graphs. For this reason, we only consider $\alpha\in (0,\frac{1}{r-1})$.
	
	Next, we extend our results about $\alpha$-maximal graphs to $\alpha$-maximal $r$-graphs. That is, we show that if $G$ is an $\alpha$-maximal $r$-graph, then the $(r-1)$-faces of $G$ have large degrees, and sets of $(r-1)$-faces have strong expansion properties.
	
	\begin{lemma} \label{lemma:hypmax}\emph{(Properties of $\alpha$-maximal simplicial complexes)}
		Let $\alpha\in (0,\frac{1}{r-1})$, and let $G$ be an $\alpha$-maximal $r$-graph. Let $p(G)=n$, and let $d(G)=cn^{\alpha}$. Then
		\begin{itemize}
			\item[(i)] If $G$ is nonempty, then $c> 1/2$.
			\item[(ii)] For every $f\in P(G)$, $\deg(f)\geq d(G)/r=cn^{\alpha}/r$.
			\item[(iii)] Let $X\subset P(G)$ such that $|X|\leq (1/2r)^{(1+\alpha)/\alpha} n$, then $|N(X)|\geq \frac{|X|}{2r}\left(\frac{n}{|X|}\right)^{\alpha/(1+\alpha)}$.
		\end{itemize}
	\end{lemma}
	
	\begin{proof}
		\begin{itemize}
			\item[(i)] If $G$ is nonempty, it contains the $r$-graph $H$ with a single edge, which satisfies $d(H)/p(H)^{\alpha}=1/r^{\alpha}>1/2$.
			\item[(ii)] Suppose that $G$ has an $(r-1)$-face $f$ such that $\deg(f)<d(G)/r$. Let $H$ be the $r$-graph we get after removing the edges containing $f$. Then $p(H)\leq p(G)-1$ and $e(H)=e(G)-\deg(f)>e(G)-d(G)/r=d(G)p(G)/r-d(G)/r$.
			Therefore, 
			$$\frac{d(H)}{p(H)^{\alpha}}=\frac{re(H)}{p(H)^{1+\alpha}}\geq \frac{ re(G)-d(G)}{(p(G)-1)^{1+\alpha}}=\frac{d(G)p(G)-d(G)}{(p(G)-1)^{1+\alpha}}>\frac{d(G)}{p(G)^{\alpha}},$$
			contradicting that $G$ is $\alpha$-maximal.
			\item[(iii)]  By (ii), each $f\in X$ is contained in at least $d(G)/r=e(G)/n$ edges, so we have 
			$$e(X\cup N(X))\geq |X|\frac{e(G)}{rn}.$$ 
			By $\alpha$-maximality, we have $e(X\cup N(X))/(|X|+|N(X)|)^{1+\alpha}\leq e(G)/n^{1+\alpha}$, which gives $$e(X\cup N(X))\leq e(G)\left(\frac{|X|+|N(X)|}{n}\right)^{1+\alpha}.$$ Comparing the lower and upper bounds on $e(X\cup N(X))$, we get 
			$$|N(X)|\geq n\left(\frac{|X|}{rn}\right)^{1/(1+\alpha)}-|X|>\frac{|X|}{2r}\left(\frac{n}{|X|}\right)^{\alpha/(1+\alpha)},$$
            where the last inequality holds by the assumption $|X|\leq (1/2r)^{(1+\alpha)/\alpha} n$.
		\end{itemize}
	\end{proof}
	
	\subsection{Expansion in simplicial complexes}\label{sect:simp_expansion}

	Our goal is to prove the following theorem, which then easily implies Theorem \ref{thm:maincycle1}.
	
	\begin{theorem}\label{thm:cycleformal}
		Let $r\geq 3$, then there exists a constant $c>0$ such that the following holds. Let $\alpha\in (0,\frac{1}{r-1})$, let $G$ be an $r$-graph with $p(G)=n$ and $e(G)\geq n^{1+\alpha}$. If $\ell>\frac{c}{\alpha}\log\frac{1}{\alpha}$ and $n>n_0(\alpha,\ell)$, then $G$ contains a cycle of length $\ell$.
	\end{theorem}
	
	In this section, we prepare the main tools needed to prove this theorem. To give some motivation, let us outline its proof.
	
	\bigskip
	
	\noindent
	\textbf{Outline of the proof of Theorem \ref{thm:cycleformal}.} We follow a similar train of thoughts as in the proof of our rainbow theorems. Let $G$ be an $r$-graph with $p(G)=n$ and $e(G)\geq n^{1+\alpha}$. We can immediately pass to an $\alpha'=\alpha/2$-maximal subhypergraph $H$ with average degree at least $n^{\alpha'}$. First, we study a random sample of $V(H)$. Similarly as in Lemma \ref{lemma:randompath_v2}, we show that if $U$ is a random subset of the vertices of $H$, where each vertex is sampled with some not too small probability, then from every $(r-1)$-face, we can reach at least $p(H)^{1-\tau}$  $(r-1)$-faces by a proper path of length $\ell_0= O(\frac{1}{\alpha}\log\frac{1}{\tau})$, whose internal vertices are in $U$. This is presented as Lemma \ref{lemma:simp_randompath}. From this, we deduce in Lemma \ref{lemma:hypmanypaths} that for every $f_0\in P(H)$, there is a set $F$ of at least $p(H)^{1-\tau}$  $(r-1)$-faces $f'$ such that there exist a large number of internally vertex disjoint proper paths of length $\ell_0$ from $f_0$ to $f'$. If $\ell$ is even, then we are immediately done as taking $\ell_0=\ell/2$, the union of two internally vertex disjoint paths between $f_0$ and $f'$ gives a cycle of length $\ell$. The case of odd $\ell$ is more difficult. We set $\ell_0=\lfloor (\ell-r-1)/2\rfloor$, and find a path $P_0$ of length $\ell-2\ell_0$, whose both endpoints $f',f''$ are in $F$, and $P_0$ is disjoint from $f_0$. If $\tau$ is sufficiently small, this exists by Lemma \ref{lemma:longhyppaths}. But then we can find two paths $P'$ and $P''$ of length $\ell_0$, such that the endpoints of $P'$ are $f_0$ and $f'$, the endpoints of $P''$ are $f_0$ and $f''$, and the three paths $P_0,P',P''$ are internally vertex disjoint. The union of $P_0,P',P''$ is a cycle of length $\ell$. This argument can be found in Lemma \ref{lemma:simp_max_cycle}, and finishes the proof.
	
	\bigskip
	
	Most of the work needed to prove Theorem \ref{thm:cycleformal} is put into the following lemma, which is a generalization of Lemma \ref{lemma:randompath_v2}. Let us introduce some notation.
	
	If $G$ is an $r$-graph and $U\subset V(G)$, a \emph{$U$-path} in $G$ is a path whose internal vertices are contained in~$U$. Also, to simplify notation, we write \wehp\  (with exponentially high probability) if some event holds with probability at least $1-\exp(-n^{z(\alpha)})$, where $z(\alpha)>0$ is some function depending only on $\alpha$.
	
	\begin{lemma}\label{lemma:simp_randompath}
		Let $r\geq 3$, $\alpha\in (0,\frac{1}{r-1})$, $\tau\in (0,1/2)$, $\ell\geq\frac{20}{\alpha}\log \frac{1}{\tau}$, then the following holds if $n>n_{0}(r,\alpha,\tau,\ell)$. Let $G$ be an $\alpha$-maximal $r$-graph with $p(G)=n$, and let $p$ be a real number such that $d(G)^{-\alpha^2/10}\leq p<1$. Let $U$ be a random sample of the vertices, each vertex chosen independently with probability $p$. Then \wehp, for every  $f\in V(G)$, there are at least $n^{1-\tau}$ $(r-1)$-faces of $G$ that can be reached from $f$ by a proper $U$-path of length $\ell$.
	\end{lemma}

	The proof of this theorem follows similar ideas as the proof of Lemma \ref{lemma:randompath}. Therefore, we first prove a variant of Lemma \ref{lemma:master}. Let us recall that in Lemma \ref{lemma:master}, we showed that if $G$ is an $\alpha$-maximal graph, then a random sample of a set $B\subset V(G)$ expands almost as well as any set of size $|B|$. The right way to generalize this is as follows. First, we have to define a somewhat unusual notion of neighborhood of $(r-1)$-faces.
	
	\begin{definition}\label{def:neighbor}
		Let $G$ be an $r$-graph, $X\subset P(G)$ and $U\subset V(G)$. Let $N_{G}(X,U)=N(X,U)$ denote the set of $(r-1)$-faces $f\in P(G)\setminus X$ such that there exists $f'\in X$ satisfying $f\cup f'\in E(G)$ and $f'\setminus f\subset U$. See Figure \ref{fig:neighbour} for an illustration.
	\end{definition}
	
	\begin{figure}
		\begin{center}
			\begin{tikzpicture}
				\node[vertex,red] (x) at (-2,0) {};
				\node at (-2,-0.5) {$\in U$};
				\node[vertex] (y) at (0.6,0.6) {};
				\node[vertex] (z) at (2,0) {};
				\node[vertex] (t) at (0,3) {};
				\node at (0,-0.5) {$f'\in X$} ;
				\node at (2.6,1.5) {$f\in N(X,U)$} ;
				
				\draw (x) -- (y) -- (z) -- (x);
				\draw (y) -- (t) -- (z);
				\draw[dotted] (x) -- (t);
			\end{tikzpicture}
		\end{center}
		\caption{An illustration for the definition of $N(X,U)$ in 4-graphs.}
		\label{fig:neighbour}
	\end{figure}
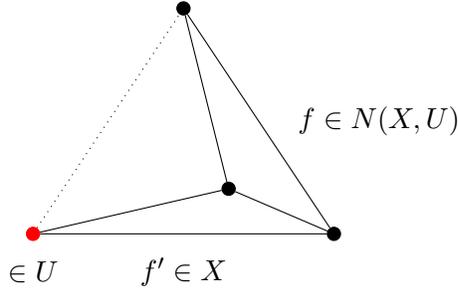

    Let us provide some explanation for the motivation of this definition. Let $X\subset P(G)$ be the set of faces we can reach by a proper path of length $\ell'\geq r-1$ from a given face $f$, such that the internal vertices of this path are from the random sample $U$. Then the set of faces we can reach by such a path of length $\ell'+1$ is a subset of $N(X,U)$. It is not necessarily equal to $N(X,U)$, as we are not allowed to repeat vertices of the path, so for this purpose, we also introduce its restricted version.

    \begin{definition}
		Let $G$ be an $r$-graph, $X\subset P(G)$ and $U\subset V(G)$. Suppose we are given a function $\phi:P(G)\rightarrow 2^{V(G)}$, which assigns each $(r-1)$-face a set of forbidden vertices. Then define $N_{\phi}(X,U)$ to be the set of those $f\in N(X,U)$ for which there exists some $f'$ satisfying the  conditions in Definition \ref{def:neighbor}, and $f\setminus f'$ is disjoint from $\phi(f')$.
	\end{definition}
 
	In what follows, if $G$ is an $r$-graph, we want to show that if $B\subset P(G)$ and $U$ is a random sample of $V(G)$, then $N(B,U)$ (or more generally $N_{\phi}(B,U)$) is large with high probability. Unfortunately, this is not true if the degree of some vertex of $B$ is too large. Therefore, we have to bound the maximum degree of $B$, and we also want to ensure that $N(B,U)$ has no large degrees either. More precisely, we find a large subset $X\subset N(B,U)$ with small maximum degree. Controlling the degrees introduces a lot of additional difficulties. Let us define a restricted version of $N(X,U)$ as well.

	To further simplify our notation, we write $a \gg b$ if $a/b>n^{z(\alpha)}$, where $z(\alpha)>0$ is some function of $\alpha$. Also, we view $r$ as a constant, so the constants hidden in the $O(.)$ and $\Omega(.)$ notation might depend on $r$.
	
	\begin{lemma}\label{lemma:simp_master}
		Let $\alpha\in (0,\frac{1}{r-1})$, $\epsilon,q\in (0,1)$, and let $r,\ell$ be positive integers, then the following holds if $n>n_{0}(r,\alpha,\ell)$. 
		\begin{itemize}
			\item  Let $G$ be an $\alpha$-maximal $r$-graph with $p(G)=n$;
			\item  let $\phi:P(G)\rightarrow 2^{V(G)}$ such that $|\phi(f)|\leq \ell$ for every $f\in P(G)$; 
			\item let $B\subset P(G)$ such that $d(G)\leq |B|< (1/10r)^{20/\alpha}n$;
			\item $\deg_{B}(v)\leq \epsilon|B|$ for every $v\in V(B)$;  
			\item suppose that $q\gg \epsilon$ and $q>d(G)^{-1/4}$.
		\end{itemize}
		Let $U$ be a random sample of $V(G)$, each vertex chosen independently with probability $q$. Then \textbf{w.e.h.p}, there exists $X\subset N_{\phi}(B,U)$ such that 
		\begin{itemize}
			\item  $|X|= |B|(n/|B|)^{\alpha/3}$, and 
			\item  $\deg_{X}(v)\leq |B|+O((\log n)|X|\epsilon/q)$ for every $v\in V(X)$.
		\end{itemize}
		
	\end{lemma}
	
	\begin{proof}
		
		Let $d:=d(G)=cn^{\alpha}$, then $d\geq n^{\alpha}/2$. As $G$ is $\alpha$-maximal, the average degree of $(r-1)$-faces in $G[B]$ is at most $c|B|^{\alpha}$. Therefore, at least half of the $(r-1)$-faces in $B$ have degree at most $2c|B|^{\alpha}$ in $G[B]$, let $C$ be the set of such $(r-1)$-faces. For $f\in C$, let $E_{f}'$ be the set of edges $e\in E(G)$ for which $e^{(r-1)}\cap B=\{f\}$ and $e\setminus f$ is disjoint from $\phi(f)$. As the minimum degree of $G$ is at least $cn^{\alpha}/r$ by Lemma \ref{lemma:hypmax}, we have $|E'_f|\geq cn^{\alpha}/r-2c|B|^{\alpha}-\ell>cn^{\alpha}/2r$. Let $E_f$ be any $cn^{\alpha}/2r=d/2r$ element subset of $E_f'$.
		
		Let $E^{*}=\bigcup_{f\in C} E_{f}$ and  $D=P(E^{*})\setminus C$, that is, $D$ is the set of $(r-1)$-faces appearing on the edges of $E_{f}$ for $f\in C$, which are not in $C$. Let us note that $|E^{*}|=|C|d/2r$. For each $e\in E^{*}$, the \emph{top-vertex} of $e$ is the unique vertex $v\in e$ such that $e\setminus \{v\}\in C$. Note that each $v\in V(G)$ is the top-vertex of at most $|C|$ edges. 
		
		Define the bipartite graph $H$ as follows.  Let the vertex classes of $H$ be $V(C)$ and $D$, and $v\in V(C)$ and $f\in D$ are joined by an edge if there exists $f'\in C$ such that $\{v\}=f'\setminus f$ and $f'\cup f\in E_{f'}$. Note that each edge $\{v,f\}$ of $H$ corresponds to the edge $\{v\}\cup f\in E^{*}$, and in the converse, each $e\in E^{*}$ corresponds to $r-1$ edges of $H$, one for each face of $e$ other than the unique face in $C$. Recalling that $U\subset V(G)$ is the random sample, we also have $N_{H}(U\cap V(C))\subset N_{\phi}(B,U)$, so it is enough to find a suitable $X\subset N_{H}(U\cap V(C))$. More precisely, our goal is to find a set $Y\subset N_{H}(U\cap V(C))$ such that $|Y|\geq 2r|B|(n/|B|)^{\alpha/3}$, and $\deg_{Y}(v)\leq r|B|+O((\log n)|Y|\epsilon/q)$ for every $v\in V(Y)$. Then a standard concentration argument shows that a random $|B|(n/|B|)^{\alpha/3}$ sized subset $X$ of $Y$ suffices with positive probability. We omit this very last concentration argument since it is routine.
		
		\bigskip
		
		For every $v\in V(C)$, we have $\deg_{H}(v)=\deg_{C}(v)d/2r$, as each $f'\in C$ containing $v$ contributes exactly $|E_{f'}|=d/2r$ edges to the degree of $v$. Say that $f\in D$ is \emph{heavy} if $\mbox{deg}_{H}(f)\geq d^{1/3}=:x$, and let $S\subset D$ be the set of heavy $(r-1)$-faces. Also, say that $e\in E^{*}$ is \emph{heavy} if all $(r-1)$-faces of $e$, except for the unique face in $C$, are heavy. Let $E_{h}\subset E^{*}$ be the set of heavy edges.
		
		Now we break our analysis into two cases, depending on whether the majority of edges of $E^{*}$ are heavy.
		
		\medskip
		\noindent
		\textbf{Case 1.} $|E_{h}|<|E^{*}|/2$. 
		
		Note that each edge $e\in E^{*}\setminus E_{h}$ corresponds to at least one edge $\{v,f\}$ of $H$ such that $f$ is not heavy. Set $T:=D\setminus S$, and consider the bipartite graph $H[T\cup V(C)]$. The previous argument implies that $$m:=e(H[T\cup V(C)])\geq |E^{*}\setminus E_{h}|\geq \frac{|C|d}{4r}.$$
		Note that $m\leq |C|d$ holds as well. For $i=0,\dots,\log_2 x$, let $T_{i}$ be the set of faces $f\in T$ such that $2^{i}\leq \deg_{H}(f)<2^{i+1}$. Let $I$ be the index $i$ maximizing the quantity $|T_{i}|2^{i}$. As 
		$$m\geq \sum_{i=0}^{\log_2 x}2^{i}|T_{i}|\geq m/2,$$
		we have 
		$$|C|d\geq m\geq |T_{I}|2^{I}\geq \frac{m}{2\log_2 x}>\frac{|C|d}{8r\log_2 n}.$$ 
		From this, $|T_{I}|\geq |C|d/8rx(\log_2 n)>|C|$.
		\begin{claim}
			For every $w\in V(T_{I})$, $\deg_{T_{I}}(w)\leq r|B|+8r\epsilon|T_{I}|\log_2 n.$
		\end{claim}
		\begin{proof}
			Let $T_{I}(w)\subset T_{I}$ be the set of $(r-1)$-faces containing $w$. Let us count the edges in the bipartite graph  $H[V(C)\cup T_{I}(w)]$ in two ways. If $\{v,f\}$ is an edge of $H$ with $v\in V(C)$ and $f\in T_{I}(w)$, then there exists $f'\in C$ such that $e:=f\cup f'\in E_{f}$ and $\{v\}=f\setminus f'$. As $w\in f$, either $w$ is a top-vertex of the edge $e$, or $w\in f'$. As each vertex of $G$ is a top-vertex of at most $|B|$ edges, we get
			$$e(H[V(C)\cup T_{I}(w)])\leq r|B|+\sum_{f'\in C, w\in f'}(r-1)|E_{f'}|\leq r|B|+\frac{\deg_{C}(w)d}{2}\leq r|B|+\epsilon|C|d.$$
			On the other hand, recalling that $2^{I}\geq |C|d/(8r|T_{I}|\log_2 n)$, we have 
			$$e(H[V(C)\cup T_{I}(w)])\geq |T_{I}(w)|2^{I}\geq |T_{I}(w)|\max\left\{1,\frac{|C|d}{8r|T_{I}|\log_2 n }\right\}.$$
			Comparing the lower and upper bound on $e(H[V(C)\cup T_{I}(w)])$, we get $|T_{I}(w)|\leq r|B|+8rt|T_{I}|\log_2 n$.
		\end{proof}
		
		Let $H':=H[V(C)\cup T_{I}]$ and recall that $U\subset V(G)$ is the random sample.
		
		\begin{claim}
			  $|N_{H'}(U\cap V(C))|\geq \frac{q|T_{I}|}{2}$ \wehp 
		\end{claim}
		\begin{proof}
			We have $m':=e(H')\geq |T_{I}|2^{I}\geq |C|d/8r\log_2 n$. Also, in the graph $H'$ the degree of each face in $T_{I}$ is between  $\Delta:=2^{I+1}$ and $\Delta/2$, and the degree of each vertex $v\in V(C)$ is at most $\deg_{C}(v)d/2r\leq \epsilon|C|d/r\leq 8\epsilon m'\log_2 n$.  That is, writing $\epsilon'=8\epsilon\log_2 n$, we have $\deg_{H'}(v)\leq \epsilon'm'$.
   
        In order to lower bound $|N_{H'}(U\cap V(C))|$, we consider a subgraph $H''$ of $H'$, in which each vertex of $T_I$ has degree 1. Then $|N_{H'}(U\cap V(C))|\geq |N_{H''}(U\cap V(C))|$, and $|N_{H''}(U\cap V(C))|$ is the sum of independent random variables. We find suitable $H''$ by assigning each face $f\in T_{I}$ to one of its neighbours $v_{f}$ randomly with uniform distribution. For each $v\in V(C)$, let $F(v)=\{f\in T_{I}:v=v_{f}\}$. Then $\mathbb{E}(|F(v)|)\leq 2\deg_{H'}(v)/\Delta\leq 2\epsilon'm'/\Delta$, so by the multiplicative Chernoff bound, we have $\mathbb{P}(|F(v)|>4\epsilon'm'/\Delta)\leq e^{-2\epsilon'm'/3\Delta}\leq e^{-\epsilon'|T_{I}|/3}=e^{-n^{\Omega(\alpha})}$. Therefore, there exists an assignment such that $|F(v)|<4\epsilon'm'/\Delta$ for every $v\in V(C)$, fix such an assignment.  Note that 
			$$|N_{H'}(U\cap V(C))|\geq \sum_{v\in U} |F(v)|.$$

   Let $X=\sum_{v\in U}|F(v)|$, then $\mathbb{E}(X)=q|T_{I}|$. Also, as $|F(v)|\leq 4\epsilon'm'/\Delta $, we can apply the multiplicative Chernoff bound with random variables $\frac{|F(v)|}{4\epsilon'm'/\Delta}\in [0,1]$ to conclude  
			$$\mathbb{P}\left(X\leq \frac{q|T_{I}|}{2}\right)\leq e^{-\Omega(q|T_{I}|\Delta/(\epsilon'm'))}.$$
			Here, $q|T_{I}|\Delta/(48\epsilon'm')=\Omega(q/\epsilon(\log n))$, so $X\geq q|T_{I}|/2$ \wehp
		\end{proof}
		Let $Y:=N_{H'}(U\cap V(C))$ and suppose that $|Y|\geq q|T_{I}|/2$. Then, for every $w\in V(Y)$, we have
		$$\deg_{Y}(w)\leq \deg_{T_{I}}(w)\leq r|B|+8r\epsilon|T_{I}|\log_2 n\leq r|B|+16r^{2}\frac{\epsilon}{q}(\log n)|Y|,$$
		so $Y$ satisfies the required maximum degree condition. Furthermore,
		$$|Y|\geq \frac{q|T_{I}|}{2}>\frac{|C|dq}{16 rx(\log_2 n)}\geq  2r|B|n^{\alpha/3}>2r|B|\left(\frac{n}{|B|}\right)^{\alpha/3},$$
		so $Y$ has the required size as well.  This concludes Case 1.
		
		\bigskip
		\noindent
		\textbf{Case 2.} $|E_{h}|\geq |E^{*}|/2$.
		
		We have $$e(G[C\cup S])\geq |E_{h}|\geq \frac{|C|d}{4r}=\frac{|C|cn^{\alpha}}{4r}.$$ 
		We show that some subset of $S$ can be chosen to be $Y$.
		\begin{claim}
			$S\subset N_{H}(U\cap V(C))$ \wehp
		\end{claim}
		
		\begin{proof}
			For a given $f\in S$, we have $\mathbb{P}(f\not\in N_{H}(U\cap V(C)))=(1-q)^{\deg_{H}(f)}<e^{-qx}$. Therefore, the probability that there exists $f\in S$ not contained in $N_{H}(U\cap V(C))$ is at most $|S|e^{-qx}=e^{-n^{\Omega(\alpha)}}$.
		\end{proof}
		
		Therefore, it remains to get rid of high degree vertices of $S$.
		To this end, we execute the following simple procedure. Let $z:=16r^2 \epsilon\log n$, $C_0:=C$, $E_{0}=E_{h}$ and $S_0:=S$. For $i=0,1,\dots$, if $C_i,E_i$ and $S_i$ are already defined, we proceed as follows. If $\deg_{S_i}(w)\leq r|B|+z|S_i|$ for every vertex $w\in V(S_i)$, we stop. Otherwise, we choose some $w_{i}$ for which $\deg_{S_i}(w_i)>r|B|+z|S_i|$, and set $C_{i+1}:=\{f\in C_{i}:w_{i}\not\in f\}$, $E_{i+1}:=\{e\in E_{i}:\exists f\in C_{i+1}\mbox{ such that }f\subset e\}$ and $S_{i+1}:=\{f\in S_{i}:\exists e\in E_{i+1}\mbox{ such that }f\subset e\}=P(E_{i+1})\setminus C_{i+1}$.
		
		Clearly, this procedure stops in a finite number of steps, and we show that if it stops at index $I$, the choice $Y:=S_{I}$ suffices. Note that for $i=0,1,\dots,I$, we have $|C_{i+1}|\geq |C_{i}|-2\epsilon|C|$ by the maximum degree condition on $B$, which also implies $|E_{i+1}|\geq |E_{i}|-2\epsilon|C|d$. Therefore, $|C_{i}|\geq (1-2\epsilon i)|C|$ and $|E_{i}|\geq (1/4r-2\epsilon i)|C|d$. Let us examine how the size of $S_{i+1}$ changes. Note that if some $f\in S_{i+1}$ contains $w_{i}$, then $w_{i}$ must be the top-vertex of some edge $e\in E_{i+1}$. Therefore, less than $r|B|$ elements of $S_{i+1}$ can contain $w_{i}$, which means that $|S_{i+1}|\leq |S_{i}|(1-z)$. This shows that $|S_{i}|\leq |S|(1-z)^{i}\leq n^{r}e^{-iz}$. But then $I\leq \frac{r\log n}{z}=\frac{1}{16r\epsilon}$, which then implies the inequality 
		$$|E_{I}|\geq \left(\frac{1}{4r}-2\epsilon I\right)|C|d\geq \frac{|C|d}{8r}.$$
		On the other hand, using the $\alpha$-maximality of $G$, we have $$|E_{I}|=e(G[C_{I}\cup S_{I}]|\leq c(|C_{I}|+|S_{I}|)^{1+\alpha}.$$
		Comparing the lower and upper bounds on $|E_{I}|$, we get 
		$$|S_{I}|\geq \left(\frac{|C|d}{8rc}\right)^{1/(1+\alpha)}-|C_{I}|\geq\frac{1}{16r}\left(|C|n^{\alpha}\right)^{1/(1+\alpha)}= \frac{|C|}{16r}\left(\frac{n}{|C|}\right)^{\alpha/(1+\alpha)}>2r|B|\left(\frac{n}{|B|}\right)^{\alpha/3}.$$
		Therefore, $Y:=S_{I}$ satisfies the required conditions. This concludes Case 2., and the proof of the theorem.
	\end{proof}
	
	Before we start the proof of Lemma \ref{lemma:simp_randompath}, we need one more lemma.
	
	\begin{lemma}\label{lemma:start}
		Let $d>r\geq 2$ be integers, and let $G$ be an $r$-graph such that $\deg(f)\geq d$ for every $f\in P(G)$. Let $f_0\in P(G)$, then there exists $F\subset P(G)$ such that 
		\begin{itemize}
			\item $|F|=(d/r)^{r-1}$,
			\item there is a proper path with endpoints $f_0$ and $f'$ of length $r-1$ for every $f'\in F$ in $G$,
			\item $\deg_{F}(v)\leq rd^{r-2}$ for every $v\in V(F)$.
		\end{itemize}
	\end{lemma}
	
	\begin{proof}
		Let $f_0=\{v_1,\dots,v_{r-1}\}$.  We generate $d(d-1)\dots(d-r+2)$ sequences of vertices $(y_1,\dots,y_{r-1})$ as follows. If $y_1,\dots,y_{i}$ are already defined for some $i\in \{0,\dots,r-2\}$, we choose $y_{i+1}$ among the vertices $y\in V(G)$ for which $\{v_{i+1},\dots,v_{r-1},y_1,\dots,y_{i},y\}$ is an edge of $G$, and $y\not\in\{v_1,\dots,v_{i}\}$; also, we choose exactly $d-i$ such vertices arbitrarily. This is possible, as by induction, we have $f:=\{v_{i+1},\dots,v_{r-1},y_1,\dots,y_{i}\}\in P(G)$, so $\deg(f)\geq d$. Let $Y$ be the set of these sequences.
		
		Clearly, if $(y_1,\dots,y_{r-1})\in Y$, then setting $f_{i}:=\{v_{i+1},\dots,v_{r-1},y_1,\dots,y_{i}\}$ for $i=1,\dots,r-1$, the sequence  $f_0,\dots,f_{r-1}$ is a proper path of length $r-1$ with endpoints $f_0$ and $\{y_1,\dots,y_{r-1}\}$. 
		
		For each sequence $(y_1,\dots,y_{r-1})\in Y$, add the face $\{y_1,\dots,y_{r-1}\}$ to $F$. Then $|F|\geq d(d-1)\dots(d-r+2)/(r-1)!>(d/r)^{r-1}$. Also, for each $y\in V(G)$ and $i\in [r-1]$, there are at most $d(d-1)\dots (d-r+2)/(d-i+1)$ sequences in $(y_1,\dots,y_{r-1})\in Y$ such that $y=y_i$. Therefore, $\deg_{F}(y)\leq rd^{r-2}$. This finishes the proof.
		
	\end{proof}

	\begin{proof}[Proof of Lemma \ref{lemma:simp_randompath}]
		In each of our calculations, we assume that $n$ is sufficiently large with respect to $r,\alpha,\ell$. Let $d=d(G)=cn^{\alpha}$, then $d\geq n^{\alpha}/2$ by Lemma \ref{lemma:hypmax}, (i). Also, by Lemma \ref{lemma:hypmax} (ii), we have $\deg(f)\geq d/r$ for every $f\in P(G)$.  Let $q\in (0,1)$ be the unique solution of $p=1-(1-q)^{\ell-r+1}$, then $p/2\ell< q< 2p/\ell$, so $d^{-\alpha^2/10}/2\ell<q<1$. Sample the vertices in $\ell-r+1$ rounds, in each round with probability $q$, independently from the other rounds, and let $U_i$ be the set of sampled vertices in round $i$ for $i\in [\ell-r+1]$. Then $U$ has the same distribution as $\bigcup_{i\in [\ell-r+1]}U_{i}$.
		
		Let  $s_{-1}:=rd^{r-2}$, $s_0:=(d/r^2)^{r-1}$, and for $i=1,\dots,\ell-r+1$, let $s_{i}:=s_{i-1}(n/s_{i-1})^{\alpha/3}$. Fix some $f_0\in P(G)$, and for $i=0,\dots,\ell-r+1$, let $\mathcal{E}_i=\mathcal{E}_i(f_0)$ denote the following event: There exist  $B_{i}\subset P(G)$ such that 
		\begin{itemize}
			\item every $f\in B_i$ can be reached from $f_0$ by a proper path of length $r-1+i$, whose internal vertices are in $U_1,\dots,U_{i}$, in this order,
			\item $|B_i|=s_{i}$,
			\item for every $v\in V(B_{i})$, we have $\deg_{B_i}(v)\leq \min\{2s_{i-1},nr^2/d\}$.
		\end{itemize}
		Here, we make the observation that $\deg_{B_i}(v)\leq \deg_{P(G)}(v)\leq nr^2/d$ holds automatically by Lemma \ref{lemma:maxdeg}. This observation is quite important, otherwise we would run into trouble when $s_{i}$ gets too close to $n$.
		
		Furthermore, note that Lemma \ref{lemma:start} guarantees that $\mathcal{E}_0$ happens. Our goal is to show that for $i=1,\dots,\ell-r+1$, we have $(\mathcal{E}_i|\mathcal{E}_{i-1},\dots,\mathcal{E}_{0})$ \wehp

		To this end, assume that $\mathcal{E}_{i-1},\dots,\mathcal{E}_{0}$ happen. Let $\phi:P(G)\rightarrow 2^{V(G)}$ be any function which assigns every $f\in B_{i-1}$ the vertices of an arbitrary path of length $i+r-2$ from $f_0$ to $f$, whose internal vertices are in $U_1,\dots,U_{i-1}$, and $\phi(f)=\emptyset$ if $f\in P(G)\setminus B_{i-1}$. Now we would like to apply Lemma \ref{lemma:simp_master} to the $r$-graph $G$ with $B=B_{i-1}$. Note that the conditions $d\leq |B|\leq (1/10r)^{20/\alpha}n$ and $|\phi(f)|\leq \ell$ for every $f\in P(G)$ are  satisfied.  Set $\epsilon:=\frac{1}{s_{i-1}}\min\{2s_{i-2},nr^2/d\}$, then we have $\deg_{B_{i-1}}(v)\leq \epsilon |B_{i-1}|$ for every $v\in V(B_{i-1})$. 
		\begin{claim}
			$n^{\alpha^2/5}\cdot \epsilon\leq q$ (that is, $q\gg \epsilon$).
		\end{claim}
		\begin{proof}
			If $i=1$, we have $\epsilon\leq 2s_{-1}/s_{0}=2r^{2r-1}/d< n^{-\alpha^2}\cdot q$. Now suppose $i\geq 2$. If $2s_{i-2}\leq nr^2/d$, we have 
			$$\epsilon=\frac{2s_{i-2}}{s_{i-1}}=2\left(\frac{s_{i-2}}{n}\right)^{\alpha/3}\leq 2\left(\frac{r^2}{d}\right)^{\alpha/3}<n^{-\alpha^2/5} d^{-\alpha/4}<n^{-\alpha^2/5} \cdot q.$$
			On the other hand, if $2s_{i-2}>nr^2/d$, then 
			$$s_{i-1}\geq \frac{nr^2}{2d}\left(\frac{2d}{r^2}\right)^{\alpha/3}>\frac{n}{d}\cdot d^{\alpha/3}.$$
			Therefore, $\epsilon=nr^2/ds_{i-1}< r^2 d^{-\alpha/3}\leq n^{-\alpha^2/5}\cdot q$.
		\end{proof}
		As the conditions of Lemma \ref{lemma:simp_master} are satisfied, \wehp, there exists $B_{i}\subset N_{\phi}(B_{i-1},U_{i})$ such that
		\begin{itemize}
			\item $|B_{i}|=|B_{i-1}|(n/|B_{i-1}|)^{\alpha/3}=s_{i}$, and
			\item $\deg_{B_{i}}(v)\leq |B_{i-1}|+O((\log n)|B_{i}|\epsilon/q)$ for every $v\in V(B_{i})$.
		\end{itemize}
		As $\deg_{B_i}(v)<nr^2/d$ holds automatically, it remains to show that $|B_{i-1}|+n^{o(1)}|B_{i}|\epsilon/q\leq 2s_{i-1}$ is also true if $|B_{i-1}|\leq nr^2/d$. If $i=1$, we have $$|B_{1}|\epsilon=\frac{2s_1 s_{-1}}{s_{0}}=\frac{2r^{2r-1}s_1}{d}=\frac{2r^{2r-1}s_{0}}{d}\left(\frac{n}{s_{0}}\right)^{\alpha/3}<s_{0}n^{-\alpha/3},$$
		so $|B_{0}|+ n^{o(1)}|B_1|\epsilon/q <2s_0$. If $i\geq 2$, then
		$$\frac{|B_{i}|\epsilon}{s_{i-1}}=\frac{2s_{i}s_{i-2}}{s_{i-1}^2}=2\left(\frac{s_{i-2}}{n}\right)^{(\alpha/3)^2}\leq 2r^2 d^{-(\alpha/3)^2}<q\cdot n^{-\alpha^3/100}.$$
		Therefore, $|B_{i-1}|+n^{o(1)}|B_{i}|\epsilon/q\leq 2s_{i-1}$ holds, as required. The set $B_{i}$ satisfies all the desired conditions, so $\mathcal{E}_i$ \wehp
		
		But then \wehp, $\mathcal{E}_{i}(f_{0})$ happens for every $i=0,\dots,\ell-r+1$ and $f_0\in P(G)$. To finish the proof, note that $s_i\geq n^{1-(1-\alpha/3)^{i}}$, so $s_{\ell-r+1}\geq n^{1-\tau}$.
	\end{proof}
	
	\subsection{Cycles in simplicial complexes --- Proof of Theorem \ref{thm:cycleformal}}\label{sect:simpcycles}
	
	In this section, we prove Theorem \ref{thm:cycleformal}. First, we use Lemma \ref{lemma:simp_randompath} to show that if $G$ is an $\alpha$-maximal $r$-graph, then between many pairs of $(r-1)$-faces, one can find $d(G)^{\Omega(\alpha^2)}$ internally vertex disjoint proper paths. We note that in order to prove the main theorem of this section, namely Theorem \ref{thm:cycleformal}, it would be enough to find $\ell$ internally vertex disjoint paths between many pairs of $(r-1)$-faces.
	
	\begin{lemma}\label{lemma:hypmanypaths}
		Let $r\geq 3$, then there exist $c_1>0$ and $0<c_2<1/4$ such that the following holds. Let  $\alpha \in (0,\frac{1}{r-1})$, $\tau\in (0,\alpha^3/30)$, $\ell\geq \frac{c_1}{\alpha}\log\frac{1}{\tau}$, and let $n>n_0(r,\alpha,\ell)$. Let $G$ be an $\alpha$-maximal $r$-graph with $p(G)=n$. Then at least $n^{2-\tau}$ pairs of faces $f,f'\in P(G)$ satisfy that there are at least $d(G)^{c_2\alpha^2}$ internally vertex disjoint proper paths of length $\ell$ with endpoints $f$ and $f'$.
	\end{lemma}
	
	\begin{proof}
		We show that $c_{1}=100$, $c_2=1/20$ suffice, say. Let $s=d(G)^{\alpha^2/10}$. We remind the reader that $d(G)>n^{\alpha}/2$ by Lemma \ref{lemma:hypmax}, (i). Partition $V(G)$ into $s$ sets $U_1,\dots,U_s$ randomly such that each vertex belongs to $U_i$ with probability $1/s$, independently from the other vertices.  
		
		Consider the multigraph $H$ on vertex set $P(G)$ in which for every $i\in [s]$, we add an edge between the $(r-1)$-faces $f$ and $f'$ if there is a proper $U_{i}$-path of length $\ell$ with endpoints $f$ and $f'$. By Lemma \ref{lemma:simp_randompath}, \wehp, every $i\in [s]$ contributes at least $\frac{1}{2}n^{2-\tau/2}>2n^{2-\tau}$ edges, so $H$ has at least $2sn^{2-\tau}$ edges. Fix a partition  $U_1,\dots,U_s$ with this property.
		
		Say that a pair of vertices $\{u,v\}$ is \emph{good} if there are at least $sn^{-\tau}$ edges between $u$ and $v$ in $H$, otherwise say that it is \emph{bad}. The bad pairs of vertices contribute at most $n^{2}\cdot sn^{-\tau}$ edges to $H$, so at least $sn^{2-\tau}$ edges of $H$ are between good pairs of vertices. As there are at most $s$ edges between any pair of vertices, this implies that there are at least $n^{2-\tau}$ good pairs of vertices. Note that each good pair of vertices $\{f,f'\}$ satisfies that there are at least $sn^{-\tau}=d(G)^{\alpha^2/10}\cdot n^{-\tau}>d(G)^{\alpha^2/20}$ internally vertex disjoint paths of length $\ell$ with endpoints $f$ and $f'$, finishing the proof.
	\end{proof}
	
	From this, we can easily deduce that sufficiently large $\alpha$-maximal $r$-graphs contain short cycles.
	
	\begin{lemma}\label{lemma:simp_max_cycle}
		Let $r\geq 3$, then there is an absolute constant $c>0$ such that the following holds. Let $r\geq 3$, $\alpha\in (0,\frac{1}{r-1})$, and let $\ell\geq \frac{c}{\alpha}\log \frac{1}{\alpha}$. If $G$ is an $\alpha$-maximal $r$-graph such that $p(G)$ is sufficiently large with respect to the parameters $r,\alpha,\ell$, then $G$ contains a cycle of length $\ell$.
	\end{lemma}
	
	\begin{proof}
		Let $\tau=\alpha^3/40$, let $\ell_0=\lfloor(\ell-r-1)/2\rfloor$, and let $c_1,c_2,n_0$ be the parameters given by Lemma \ref{lemma:hypmanypaths}. Choose $c$ such that that $\ell_0\geq  \frac{c_1}{\alpha}\log\frac{1}{\tau}$ holds. Let $n=p(G)$, $d=d(G)>n^{\alpha}/2$, and assume that $n>n_0(r,\alpha,\ell_0)$ and $n^{c_2\alpha^3}\geq 2\ell$. Say that a pair of $(r-1)$-faces $\{f,f'\}$ is good if there are at least $d^{c_2\alpha^2}>\ell$ internally vertex disjoint proper paths of length $\ell_0$ with endpoints $f$ and $f'$. Then there are at least $n^{2-\tau}$ good pairs, so there exists $f_0\in P(G)$ which appears in at least $n^{1-\tau}$ good pairs. Let $F\subset P(G)$ be the set of faces $f$ such that $\{f_0,f\}$ is good. Let $H$ be the subhypergraph of $G$ we get after removing the vertices (and thus all edges containing some of those vertices) of $f_0$. By Lemma \ref{lemma:hypmax}, every $f\in P(G)$ satisfies $\deg_{G}(f)\geq d/r$, so every $f\in P(H)$ satisfies $\deg_{H}(f)\geq d/r-(r-1)>d/2r$. Also, as $f\cap f_0=\emptyset$ for every $f\in F$, we have $F\subset P(H)$. 
		
		Apply Lemma \ref{lemma:longhyppaths} to the $r$-graph $H$ and $F\subset P(H)$. We have $\ell-2\ell_0\leq r+2$ and $|F|\geq n^{1-\tau}>4r^2(r+2)p(H)/d$, so $H$ contains a proper path of length $\ell-2\ell_0$, whose both endpoints are in $F$. Let $f_{\ell_0},\dots,f_{\ell-\ell_0}$ be such a path. As $f_{\ell_0}\in F$, there exist at least $\ell>r+2$ internally vertex disjoint paths of length $\ell_0$ with endpoints $f_0$ and $f_{\ell_0}$. Therefore, there exists one such path whose every internal vertex is disjoint from the faces $f_{\ell_0},\dots,f_{\ell-\ell_0}$. Also, $f_{\ell-\ell_0}\in F$, so there exist at least $\ell$ internally vertex disjoint paths of length $\ell_0$ with endpoints $f_0$ and $f_{\ell-\ell_0}$. But then there exists one whose every internal vertex is disjoint from the faces $f_0,\dots,f_{\ell-\ell_0}$. Let $f_{\ell-\ell_0},f_{\ell-\ell_0+1},\dots,f_{\ell}=f_{0}$ be such a path. To finish the proof, note that $f_0\dots,f_{\ell}$ is a cycle of length $\ell$ in $G$. Indeed, $f_{\ell_0}$ and $f_{\ell-\ell_0}$ are disjoint, and the two paths $f_{\ell_0},\dots,f_{\ell-\ell_0}$ and $f_{\ell-\ell_0},\dots,f_{\ell_0}$ (where indices are meant modulo $\ell$) are internally vertex disjoint.
	\end{proof}
	
	Now everything is set to prove the main theorem of this section.
	
	\begin{proof}[Proof of Theorem \ref{thm:cycleformal}]
		Let $c'>0$ be the constant $c$ given by Lemma \ref{lemma:simp_max_cycle}. Let $G'$ be the subhypergraph of $G$ maximizing the quantity $d(G')/p(G')^{\alpha/2}$. Then $G'$ is $(\alpha/2)$-maximal and $d(G')\geq n^{\alpha/2}$. Therefore, $p(G')\geq n^{\alpha/2}$. Hence, if $\ell\geq \frac{c'}{\alpha/2}\log \frac{1}{\alpha/2}$, and $n$ is sufficiently large with respect to $r,\alpha,\ell$, then $G'$ contains a cycle of length $\ell$. This shows that $c=4c'$ suffices.
	\end{proof}
	
	Finally, we deduce Theorem \ref{thm:maincycle1}. Let us repeat this theorem formally.
	
	\begin{theorem}\label{thm:final}
		There exists a constant $c>0$ such that the following holds. Let $\alpha\in (0,1/2)$, let $G$ be a $3$-graph with $p(G)=n$ and $e(G)\geq n^{1+\alpha}$. If $n>n_{0}(\alpha)$, then $G$ contains a homeomorphic copy of the cylinder and the M\"obius strip on at most  $\frac{c}{\alpha}\log\frac{1}{\alpha}$ vertices.
	\end{theorem}
	
	\begin{proof}
		By a standard probabilistic argument, $G$ contains a 3-partite subhypergraph $G'$ with at least $2e(G)/9$ edges. Applying, Theorem \ref{thm:cycleformal} to $G'$, we get that $G'$ contains an $\ell$-cycle for every $\ell>\frac{c'}{\alpha}\log\frac{1}{\alpha}$ for some constant $c'>0$. Recall from Lemma \ref{lemma:characterization} that a 3-partite $\ell$-cycle is homeomorphic to the cylinder if $\ell$ is even, and to the M\"obius strip if $\ell$ is odd. This finishes the proof.
	\end{proof}

	\section{Concluding remarks}
	
	\subsection{Cycles in the hypercube}
	
	Theorem \ref{thm:cycleformal} has an interesting application about the extremal numbers of cycles in the hypercube. As a reminder, $Q_n$ denotes the graph of the $n$-dimensional hypercube, that is, $V(Q_n)=\{0,1\}^{n}$, and two vertices are joined by an edge if they differ in exactly one coordinate. In other words,  each edge of $Q_n$ corresponds to a sequence $\{0,1,*\}^{n}$ containing exactly one $*$, called the \emph{flip-bit}.
	
	Given a graph $H$, let $\mbox{ex}(Q_n,H)$ denote the maximum number of edges of a subgraph of $Q_n$ containing no copy of $H$. The study of the function $\mbox{ex}(Q_n,C_{2\ell})$ was initiated by Erd\H{o}s \cite{Erdos84,Erdos90}; we refer the reader to \cite{Conlon10} for the extensive history of this problem. In case $\ell\in\{2,3\}$, one has $\mbox{ex}(Q_n,C_{2\ell})=\Omega(n\cdot 2^{n})$, while for $\ell\geq 4$ the Lov\'asz local lemma implies that $\mbox{ex}(Q_n,C_{2\ell})=\Omega(n^{\frac{1}{2}+\frac{1}{2\ell}}\cdot2^{n})$, see \cite{Conlon10}. For an upper bound, Chung \cite{Chung92} proved that if $\ell\geq 4$ is even, then $\mbox{ex}(Q_n,C_{2\ell})=O( n^{\frac{1}{2}+\frac{1}{\ell}}\cdot2^{n})$. The case of odd $\ell$ is more mysterious. It was proved only later by F\"uredi and \"Ozkahya \cite{FurOz1,FurOz2} that for $\ell\geq 7$, there exists $\epsilon(\ell)>0$ such that $\mbox{ex}(Q_n,C_{2\ell})=O(n^{1-\epsilon(\ell)}\cdot 2^{n})$, where $\epsilon(\ell)$ tends to $1/16$ as $\ell$ tends to infinity. Pursuing a remark of Conlon \cite{Conlon10}, who also gave a short proof of this (with slightly weaker bounds on $\epsilon(\ell)$), we improve this for large $\ell$.
	
	\begin{theorem}\label{thm:hypercube}
		$\mbox{ex}(Q_n,C_{2\ell})=O(n^{\frac{2}{3}+\delta}\cdot 2^{n})$ for some $\delta=O(\frac{\log \ell}{\ell})$.
	\end{theorem}
	
	If $\mathcal{G}$ is a family of $r$-graphs, let $\mbox{ex}(n,\mathcal{G})$ denote the maximum number of edges in an $r$-graph on $n$ vertices containing no member of $\mathcal{G}$, and if $\mathcal{G}=\{G\}$, write $\mbox{ex}(n,G)$ instead of $\mbox{ex}(n,\mathcal{G})$.  
	
	Furthermore, we use the following terminology and results from \cite{Conlon10}. We say that an $r$-graph $G$ with vertex set $[m]$ is a \emph{representation} of a graph $H$ if the following holds: $H$ has a copy in $Q_{m}$ such that each edge $\{a_1,\dots,a_{m}\}\in \{0,1,*\}^{m}$ has exactly $r$ nonzero entries ($r-1$ ones and one flip-bit), and if $i_1,\dots,i_r$ are the indices of the nonzero entries, then $\{i_1,\dots,i_r\}\in E(G).$
	
	\begin{theorem}\cite{Conlon10}
		Let $H$ be a graph and $G$ be an $r$-graph representing $H$. Suppose that $\mbox{ex}(n,G)\leq \epsilon n^{r}$. Then $$\mbox{ex}(Q_{n},H)=O(\epsilon^{1/r}n\cdot 2^{n}).$$ 
	\end{theorem}
	
	Moreover, the following slightly more general result is also true and follows from the same proof.
	
	\begin{theorem}\label{thm:Conlon}
		Let $H$ be a graph and let $\mathcal{G}$ be a family of $r$-graphs, each of which is a representation of $H$. Suppose that $\mbox{ex}(n,\mathcal{G})\leq \epsilon n^{r}$. Then $$\mbox{ex}(Q_{n},H)=O(\epsilon^{1/r}n\cdot 2^{n}).$$ 
	\end{theorem}
	
	\begin{proof}[Proof of Theorem \ref{thm:hypercube}]
		Let $\mathcal{C}$ be the family of 3-graphs associated with cycles of length $\ell$. Note that each member of $\mathcal{C}$ is a representation of $C_{2\ell}$. By Theorem \ref{thm:cycleformal}, we have $\mbox{ex}(n,\mathcal{C})=n^{2+O(\log \ell/\ell)}$. But then Theorem \ref{thm:Conlon} gives $\mbox{ex}(Q_{n},C_{2\ell})=O(n^{2/3+O(\log \ell/\ell)}\cdot 2^{n}).$
	\end{proof}
	
\subsection{Open problems}
	
	In this paper, we managed to improve the best known upper bounds on the extremal numbers of rainbow cycles and subdivisions, but there are still some gaps to overcome.
	\begin{itemize}
		\item Show that if a  properly edge colored graph with $n$ vertices contains no rainbow cycles, then its average degree is $O(\log n)$.
		
		\item Let $t\geq 3$ be an integer. Show that if a properly edge colored graph with $n$ vertices contains no rainbow subdivision of $K_{t}$, then its average degree is $O_t(\log n)$.
	\end{itemize}
	Furthermore, we believe the following strengthening of Theorem \ref{thm:robust} should be also true.
	\begin{itemize}
		\item Let $G$ be a graph with $n$ vertices and $n^{1+\alpha}$ edges. Show that $G$ contains an $\ell$-subdivision of $K_t$, where $t=n^{\Theta(\alpha)}$ and $\ell=O(1/\alpha)$. 
	\end{itemize}

    \noindent
    \textbf{Updates.} After the submission of this paper, there have been several developments regarding the problems above. O. Janzer and Sudakov \cite{JS22}, and independently Kim, Lee, Liu, and Tran \cite{KLLT22} proved that the extremal number of rainbow cycles is $O(n(\log n)^2)$, both proofs relying on a clever counting argument. Also, by slightly modifying the argument of our paper, Wang \cite{W22} proved  that the extremal number of rainbow subdivisions of $K_t$ for every fixed $t$ is at most $n(\log n)^{2+o(1)}$.
	
	\section*{Acknowledgments}
	We would like to thank Matija Buci\'c, Oliver Janzer, Tao Jiang, J\'anos Nagy, Sam Spiro, and Benny Sudakov for valuable discussions, and the anonymous referees for their useful comments and suggestions. 
	
	Part of this work was done while the author was employed at ETH Zurich, where he was supported by the SNSF grant 200021\_196965.
 
	\bibliographystyle{amsplain}

 \begin{aicauthors}
\begin{authorinfo}[ist]
  Istv\'an Tomon\\
  Ume\r{a} University\\
  Ume\r{a}, Sweden\\
  istvantomon\imageat{}gmail\imagedot{}com \\
  \url{https://sites.google.com/view/istvantomon/home}
\end{authorinfo}

\end{aicauthors}
	
\end{document}